\documentclass[12pt,leqno,dvipdfmx]{amsart} 
\usepackage{euscript, amssymb, amsmath, amsthm}
\usepackage{epsfig}
\usepackage{graphicx}
\usepackage{caption}
\usepackage{longtable}
\usepackage{dcolumn}
\usepackage{color}

\usepackage{mathtools}
\mathtoolsset{showonlyrefs=true}

\usepackage{amsfonts}

\newcommand\smallO[1]{
	\mathchoice
		{
			\ensuremath{\mathop{}\mathopen{}{\scriptstyle\mathcal{O}}\mathopen{}\left(#1\right)}
		}
		{
			\ensuremath{\mathop{}\mathopen{}{\scriptstyle\mathcal{O}}\mathopen{}\left(#1\right)}
		}
		{
			\ensuremath{\mathop{}\mathopen{}{\scriptscriptstyle\mathcal{O}}\mathopen{}\left(#1\right)}
		}
		{
			\ensuremath{\mathop{}\mathopen{}{o}\mathopen{}\left(#1\right)}
		}
}

\usepackage{setspace}
\usepackage[most]{tcolorbox}
\definecolor{webred}{rgb}{0.75,0,0}
\definecolor{webgreen}{rgb}{0,0.75,0}
\usepackage[colorlinks=true, citecolor=blue]{hyperref}

\numberwithin{equation}{section}

\newtheorem{theo}{Theorem}[section]
\newtheorem{lem}{Lemma}[section]
\newtheorem{prop}[theo]{Proposition}
\newtheorem{Def}[theo]{Definition}
\theoremstyle{remark}
\newtheorem{rem}{Remark}[section]

\newcommand{\f}{\Omega}
\newcommand{\g}{\Gamma}

\def\y{{\mathrm{d}}y}

\def\R{{\mathbb{R}}}
\def\N{{\mathbb{N}}}

\def\E{{\mathbb{E}}}
\def\LL{{R_0(t)}}
\def\L{{R(t)+1}}

\tcbset{
    frame code={},
    center title,
    left=0pt,
    right=0pt,
    top=0pt,
    bottom=0pt,
    colback=gray!10,
    colframe=white,
    width=\dimexpr\textwidth\relax,
    enlarge left by=0mm,
    boxsep=5pt,
    arc=0pt,outer arc=0pt,
}

\begin{document}

\title{Asymptotic profiles for the Cauchy problem of semilinear  beam equation with two variable coefficients in the subcritical case}

\author[M. A. Hamza]{Mohamed Ali Hamza${}^{1}$}
\thanks{${}^{1}$
Basic Sciences Department,
Deanship of Preparatory Year and Supporting Studies,
P. O. Box 1982, Imam Abdulrahman Bin Faisal University,
Dammam, KSA%
}
\email{mahamza@iau.edu.sa}

\author[Y. Wakasugi]{Yuta Wakasugi${}^2$}
\thanks{${}^2$
Laboratory of Mathematics,
Graduate School of Advanced Science and Engineering,
Hiroshima University,
Higashi-Hiroshima, 739-8527, Japan%
}
\email{wakasugi@hiroshima-u.ac.jp}

\author[S. Yoshikawa]{Shuji Yoshikawa${}^2$}
\email{s-yoshikawa@hiroshima-u.ac.jp}
\date{\today}

\medskip

%
%
%
%

\pagestyle{plain}


\maketitle
\begin{abstract}
In this article, we investigate the asymptotic profile of solutions to
the Cauchy problem for a nonlinear beam equation with two variable
coefficients in the subcritical nonlinear case.
In contrast to our previous result \cite{our1}, in which the asymptotic profile
is governed by the linear heat kernel and the nonlinear effect is
asymptotically negligible, the asymptotic profile in the present setting
is described by a self-similar solution to the associated nonlinear
parabolic equation (constructed in Brezis-Peletier-Terman \cite{BPT}).
The proof relies on delicate energy estimates in weighted spaces
formulated in parabolic self-similar variables.
\end{abstract}


\section{Introduction  and statement of the main result}
We consider the semilinear beam equation with two variable coefficients
\begin{equation}
\label{A}
	\partial_{tt}u +b(t) \partial_{t} u -a(t) \partial_{xx}u + \partial_{xxxx}u=-|u|^{p-1}u, 
	\quad  t\in (0,\infty), \ x\in \R,
\end{equation}
where
$u= u(t,x)$
is a real-valued unknown, $a(t)$ and $b(t)$
are positive functions of $t$,
and $p>1$. 

The equation corresponds to the damped fourth order wave equation with two variable coefficients. 
In \cite{le}, Levandosky studies the decay estimate for the Cauchy problem for the fourth order wave equation 
\[
\partial_{tt} u + \Delta^2 u + u = f(u), \qquad  |f'(u)| \lesssim |u|^{p-1}. 
\]
From the physical point of view, most of the nonlinear terms for the semilinear beam equation take the derivative nonlinear term 
$\partial_x f(\partial_x u)$
because in nonlinear elasticity the stress is usually expressed as a function of the strain.   
The nonlinear term of the Boussinesq type equation is
$\partial_{xx} f(u)$ (e.g. \cite{fa-gr}),
which also includes the derivative. 
For this reason, in the authors' previous result \cite{our1}, the problem with the derivative nonlinear term
$\partial_x (|\partial_x u|^{p})$
was studied, but the critical power does not appear there, as is also easily checked by the scaling argument. 
On the other hand, as mentioned later, it is interesting to investigate the problem with the nonlinearity 
$f(u)$, because the critical power from the scaling argument appears in the case. 
From the mathematical point of view, we investigate the problem \eqref{A}. 

The purpose of this paper is to study the asymptotic behavior of the solution to the equation \eqref{A}.
Evidently, the asymptotic behavior of the solution
will change depending on the coefficients $a(t)$ and $b(t)$, and the parameter $p$.
In particular, in this paper, we investigate the case where
the asymptotic properties of the solution are determined by the corresponding nonlinear parabolic regime.

To fix our setting, we introduce the assumptions on the coefficients
$a(t)$ and $b(t)$, and the power $p$ in the following.

First, we assume that the coefficients $a(t)$ and $b(t)$ 
are smooth positive functions
satisfying
\begin{equation}\label{ab}
|a(t)  -(1+t )^{\alpha}|\le C (1+t)^{\alpha-1}, \quad 
 |b(t)-  b_0(1+t )^{\beta}|\le C(1+t)^{\beta-1},
 \end{equation}
  for some   $\alpha, \beta \in \R$ and $b_0>0$.
Further, we assume 
  that  their derivatives  $a'$ and $b'$
  satisfy  the conditions:
\begin{equation}\label{H1} 
	\big|a'(t)\big| \le  C(1+t)^{\alpha-1}\qquad \text{and} \qquad  \big|b'(t)\big| \le   C (1+t)^{\beta-1}.
\end{equation}

Under these assumptions on $a(t)$ and $b(t)$,
it is proved in \cite{YW1} that
for the linearized problem of \eqref{A} and for the parameters
$(\alpha, \beta)$ satisfying
$-1<\beta<\min(2\alpha +1, \alpha+1)$,
the asymptotic behavior of the solution is described by the corresponding parabolic equation
$b_0 t^{\beta} \partial_t v - t^{\alpha} \partial_{xx} v = 0$.
Moreover, the other parameter region for $(\alpha, \beta)$ was classified into four regions
corresponding to wave-type, beam-type, fourth-order heat-type, and overdamping-type asymptotic behaviors.

In this paper, we consider the case
$-1<\beta<\min(2\alpha +1, \alpha+1)$,
where the expected asymptotic behavior of the linear part
is given by the associated parabolic equation:

\noindent
\textbf{Assumption (I)}
 \begin{equation*}
 (\alpha,\beta)\in \mathcal{A} :=\{ (\alpha,\beta)\in \R^2; -1<\beta< \min (2\alpha+1, \alpha+1)\}.
 \end{equation*}

For readers' convenience, we explain the meaning of this assumption
by giving a formal scaling argument, which indicates the asymptotic behavior of the solution to \eqref{A}
for $(\alpha, \beta) \in \mathcal{A}$.
In order to get  the leading terms,
we first replace
$a(t)$ and $b(t)$ in the linearized equation of \eqref{A}
with their asymptotic equivalents as $t\to \infty$ given in \eqref{ab}.
Then, we formally obtain
\begin{equation}
\label{A123L}
	\partial_{tt}u+b_0t^{\beta} \partial_{t} u -t ^{\alpha} \partial_{xx}u + \partial_{xxxx}u = 0.
\end{equation}
Now, for $\lambda>0$, we define
\begin{equation}\label{transf}
	u(t,x)=v(\lambda^{\nu} t,\lambda x), \qquad \lambda^{\nu} t=s, \quad \lambda x=y,
\end{equation}
for some $\nu>0$.
Clearly, if $u$ is a solution of  the associated  linear problem  to \eqref{A123L}, then $v$ satisfies
\begin{equation}
\label{A123bbL}
	\lambda^{\nu (1+\beta)}\partial_{ss}v+b_0s^{\beta} \partial_{s} v
	-\lambda^{2-(\alpha-\beta+1)\nu}s^{\alpha} \partial_{yy}v
	+\lambda^{4-\nu+\beta \nu} \partial_{yyyy}v
	=0.
\end{equation}
Thus, if we set $\nu=\frac2{\alpha-\beta+1}$ and
use the condition $(\alpha, \beta) \in \mathcal{A}$,
then we infer
$1+\beta >0$ and $4+\nu(\beta-1)=\frac{2(1+2\alpha-\beta)}{\alpha-\beta+1}>0$.
Therefore, by letting $\lambda\rightarrow 0$, 
we keep only the leading order terms, namely   
\begin{equation}\label{b0}
	b_0s^{\beta} \partial_{s} v -s^{\alpha} \partial_{yy}v =0.
\end{equation}
The kernel of the above equation is given by
\begin{equation}
K(s,y)= \frac{1}{\sqrt{4\pi R_0(s)}} \exp \left( - \frac{y^{2}}{4R_0(s)} \right),
\end{equation}
where
\begin{equation}\label{transf0}
 R_0(s):=\frac{s^{\alpha-\beta+1}}{b_0(\alpha-\beta+1)},
\end{equation}
and it can be easily seen that
$K(s,y)$ gives the asymptotic profile of solutions to the equation \eqref{b0}.

By the modification that
$u(t,x)=\lambda^{\zeta} v(\lambda^{\nu} t,\lambda x)$ in \eqref{transf}
with appropriately chosen $\zeta>0$,
we can include the nonlinear term $-|u|^{p-1}u$ in the above scaling argument.
If the nonlinearity is supercritical, that is,
$p>1+\frac{2(1-\beta)}{\alpha-\beta+1}$,
then one can see that the nonlinear term is still negligible and the same equation as \eqref{b0} is obtained.
This also suggests that the asymptotic profile is still given by $K(s,y)$ in the supercritical case.
Indeed, in our previous result \cite{our1}, the equation \eqref{A} with the derivative nonlinearity
$\partial_x (|\partial_x u|^{p-1} \partial_x u)$ was studied,
and the asymptotic profile is determined by $K(s,y)$.
The equation \eqref{A} can be handled in the same way as \cite{our1} in the supercritical case
$p>1+\frac{2(1-\beta)}{\alpha-\beta+1}$,
and we see that the asymptotic profile is given by $K(s,y)$.

However, in the subcritical case, the situation is different,
and the nonlinear effect should be involved in describing the asymptotic behavior. 
The aim of the present paper is to deal with the subcritical case, namely, $p$ satisfies:

\noindent
\textbf {Assumption (II)}
\begin{equation}\label{scritical}
1<p<1+\frac{2(1-\beta)}{\alpha-\beta+1}.
\end{equation}

Now, we construct a particular solution of the corresponding semilinear parabolic equation obtained by
substituting $a(t)$ and $b(t)$ by their equivalent forms thanks to \eqref{ab} and ignoring the lower terms
$\partial_{tt}u$ and 
$ \partial_{xxxx}u$
in \eqref{A}.
More precisely, we consider the equation:
\begin{equation}
\label{A1}
b_0(1+t )^{\beta} \partial_{t} u -(1+t )^{\alpha}
\partial_{xx}u =-|u|^{p-1}u, 
\quad t \in (0,\infty), \ x \in  \R.
\end{equation}
 We define
\begin{equation}\label{transf0}
	V(T,x)=u(t,x), 
	\quad
	\text{where} \quad
	T:=R_0(t+1),
\end{equation}
where the function $R_0$ is defined in \eqref{transf0}.
It is easy to verify that 
if $u(t,x)$ is a solution  of \eqref{A1}, then $V(T,x)$ satisfies
\begin{equation}\label{A2}
	\partial_{T} V - \partial_{xx}V +\frac{|V|^{p-1}V}{\Big(b_0(\alpha-\beta+1)T\Big)^{\alpha/(\alpha-\beta+1)}}=0, 
	\quad T\in [R_0(1),\infty), \ x\in \R.
\end{equation}
A self-similar  solution of the equation \eqref{A2} is given by
\[
	V(T,x)=\frac{A_0}{T^{\theta} }\f \left( \frac{x}{\sqrt{T}} \right),
\]
where 
\begin{equation}\label{A0}
\theta :=\frac{1-\beta}{(1-\beta+\alpha)(p-1)},
\qquad
A_0:= (b_0(\alpha-\beta +1))^{\frac{\alpha}{(\alpha-\beta+1)(p-1)}},
\end{equation}
and  $\f$ satisfies
the following  ordinary differential equation
\begin{equation}\label{1.4}
	\f^{''}(z)+\frac{z}{2}\f^{'}(z)+\theta  \f(z)
	-|\f(z)|^{p-1}\f(z)=0,
	\quad z \in \R.
\end{equation}
Thanks to \cite[Theorems 3 and 4]{BPT}, there exists a positive solution of \eqref{1.4}.
More precisely, for all $c_0>0$
there exists a smooth, even, and  positive
solution $\f$
 of the equation (\ref{1.4}) such that  $\lim_{|z|\to \infty}|z|^{2\theta}\f (z)
=c_0$ and  $\f'(0)=0$
(see Lemma \ref{L3bis1} for more detailed behavior of $\Omega$).

Throughout the paper,
we denote
\begin{equation}\label{def:gamma}
\g :=\g(t,x)=\frac{A_0}{{(\LL)}^ {\theta }} \f \left( \frac{x}{\sqrt \LL} \right).
\end{equation}

\par
If $(\alpha, \beta)$ and $p$ satisfy \textbf{Assumptions (I), (II)}, 
then $\g$ is expected to be an asymptotic solution of \eqref{A}.
Although $\g$ is not a solution of equation (\ref{A}), we can study its asymptotic stability in the following sense.
Let $t_0>1$ be a fixed, sufficiently large real number that will be chosen later;
for the initial data
$(u(t_0,x),\partial_tu (t_0,x))$ near
$(\g(t_0,x),\partial_{t}
\g(t_0,x))$,
 the corresponding solution $u(t,x)$ of equation
(\ref{A}) converges to
 $\g(t,x)$
 in an appropriate norm,
 when $t \to +\infty$.

To study the asymptotic behavior of the solutions of equation \eqref{A}
around $\Gamma(t,x)$ defined above,
we introduce the following scaling variables:
\begin{equation}\label{scaling0}
y=\frac{x}{\sqrt{\L}}, \qquad   \textrm{and} \qquad s=\log(\L),
\end{equation}
where 
\begin{equation}\label{defR}
 R(t)=\int_0^tr(\tau) \,\mathrm{d}\tau, \qquad r(t)=\frac{a(t)}{b(t)}.\ \  
\end{equation}
Writing the equation \eqref{A} as a first-order  system in the variable  $(u,\partial_tu)$ and rescaling
the two components independently, we are  in position  to set,
\begin{align}\label{scaling1}
	u(t,x) &= \frac{A_0}{(\L)^{\theta }} v \left( \log(\L), \frac{x}{\sqrt{\L}} \right),\\
	\partial_tu(t,x) &= \frac{A_0r(t)}{(\L)^{\theta +1 }}w \left( \log(\L), \frac{x}{\sqrt{\L}} \right),
\end{align}
A straightforward computation implies  that $u$  is a solution of \eqref{A} if and only if $(v,w)$
satisfies the following:
\begin{equation}
\label{vw} \left\lbrace
\begin{alignedat}{1}
	&w = v_s-\frac{y}2v_y-\theta  v,\\
	&\frac{r^2e^{-s}}{a}\Big(w_s-\frac{y}2w_y-(\theta +1)  w\Big)+\frac{r'}{a}w
	+w-v_{yy}+\frac{e^{-s}}{a}v_{yyyy}=-|v|^{p-1}v+\varepsilon(t)|v|^{p-1}v,
\end{alignedat}
\right.
\end{equation}
where
$\varepsilon(t)$
is
\begin{equation}\label{varepsilon}
\varepsilon (t)=1-\frac{(b_0(\alpha-\beta+1))^{\frac{\alpha}{\alpha-\beta+1}}(1+R(t))^{\frac{\alpha}{\alpha-\beta+1}}}{a(t)}
= \mathcal{O} \left( (1+t)^{-1} \right)
\end{equation}
under hypothesis \eqref{ab},
which means that the term
$\varepsilon(t) |v|^{p-1}v$
can be regarded as a perturbation.

In order to study the asymptotic stability of $\Gamma(t,x)$
for the equation \eqref{A},
we further rewrite the system of $(v,w)$ into that of the perturbations around $\Omega(y)$.
To this end, we introduce
\begin{equation}\label{1.14}
	f(s,y)=v(s,y) -\f(y) \quad \text {and} \quad g(s,y) =w(s,y) +\theta \f(y) +\frac{y}{2}\f' (y).
\end{equation}
Then, $f$ and $g$ satisfy the equations
\begin{equation}\label{scaling}
	\left\{
	\begin{alignedat}{1}
	&g=f_{s}-\frac{y}{2}f_{y}-\theta  f,\\
	&\frac{r^2e^{-s}}{a}\Big(g_s-\frac{y}2g_y-(\theta  +1) g\Big)+\frac{r'}{a}g
		+f_s=-\frac{e^{-s}}{a}f_{yyyy}+{\mathcal  L}(f)+{\mathcal N}(f)+h,
	\end{alignedat}
	\right.
\end{equation}
where
\begin{equation}\label{1.17}
	\begin{alignedat}{3}
	{\mathcal L}(f)&:=
		f_{yy}+\frac{y}{2}f_{y} +\theta  f-p\f^{p-1}f, 
	\\
	{\mathcal N}(f) &:=
		-\left( |f+\f|^{p-1}(f+ \f) - \f^{p}(y)-p \f^{p-1}f \right),
	\\
	h(s,y) &:=
		h_1(s,y)+h_2(s,y),\\
	h_1(s,y) &:=
		- \frac{r^2e^{-s}}{a}
			\left(
				(\theta +1)(\theta \f +{\frac{y}{2}}\f')+\frac{y}{2}((\theta +\frac{1}{2})\f'+\frac{y}{2}\f'')
			\right)
		\\
	&\quad +\frac{r'}{a}(\theta \f +\frac{y}{2}\f')-\frac{e^{-s}}{a}\f'''',
	\\
	h_2(s,y) &:=
		+\varepsilon (t)|f+ \f|^{p-1}(f+ \f).
	\end{alignedat}
\end{equation}
At this stage, the asymptotic stability of $\Gamma(t,x)$
reduces to proving that
$(f,g)$ decay to zero as $s\to \infty$ with small initial data at $s=s_0$
in an appropriate space.

We now  introduce the Hilbert spaces in which we are going to study our problem (\ref{scaling}).
For $\ell\ge 1$, we define $L^{2,\ell}(\R)$ as the weighted
Lebesgue space:
\[
	L^{2,\ell}(\R)=\Big\{u\in L^2(\R), (1+|y|^{ \ell})u\in L^2(\R)\Big\},
	\quad
	\|u\|^2_{L^{2, \ell}(\R)}=\int_{\R} (1+|y|^{2 \ell}) u^2 \, \mathrm{d}y.
\]
Also, for $k\in \N$, we define  the following weighted Sobolev spaces:
\[
	H^{k, \ell}(\R)=\Big\{u\in  L^{2, \ell}(\R),\ \ \partial^k_yu\in  L^{2, \ell}(\R)\Big\},
	\quad
	\|u\|^2_{ H^{k, \ell}(\R)}=\sum_{i=0}^k\|\partial_y^iu\|^2_{ L^{2, \ell}(\R)}.
\]
We sometimes use the notation $H^k(\mathbb{R}) := H^{k,0}(\mathbb{R})$.
In particular, we set
$X^{\ell}(\R)=H^{2,\ell}(\R)\cap H^{3,\ell-1}(\R)$
and
$\ Y^{\ell}(\R)=H^{0,\ell}(\R)\cap H^{1,\ell -1}(\R)$.
We define the product space
$Z^{\ell}(\R)=X^{\ell}(\R)\times Y^{\ell}(\R)$ 
endowed with the natural norm:
\begin{equation}
\|(v,w)\|^{2}_{Z^{\ell}(\R)}=\|v\|^{2}_{X^{\ell}(\R)} + \|w\|^{2}_{Y^{\ell}(\R)}.
\end{equation}

To study the system \eqref{scaling},
we work on the estimate of solutions in the space $Z^1(\mathbb{R})$
by applying the weighted energy method.
Before stating our main result,
we have to impose one more assumption on the parameters $\alpha, \beta, p$.

\noindent  
\textbf{Assumption (III)}
\begin{equation*}
\begin{cases}
	1+\frac{4(1-\beta )}{3(\alpha-\beta+1)}<p <\frac1{1-\frac{4(1-\beta )}{3(\alpha-\beta+1)}}
 	& \text{if} \quad \theta(p-1)< \frac14,\\
	1+\frac{4(1-\beta )}{3(\alpha-\beta+1)}<p 
	&\text{otherwise} .
\end{cases}
\end{equation*}
Although \textbf{Assumption (III)} is technical, its restrictions appear to be necessary.
Indeed, the left-hand inequalities of
\textbf{Assumption (III)}
are equivalent to $\theta < \frac{3}{4}$,
which, as we will see later in Remark \ref{rem:spec},
guarantees that the essential spectrum of $\mathcal{L}_{\infty}$ lies strictly in the left half-plane:
$\sigma_{\mathrm{ess}}(\mathcal{L}_{\infty}) \subset
\left\{ z \in \mathbb{C} \;\middle|\; \text{Re}(z) \le -\tilde \lambda\ \right\}$,
for some $\tilde \lambda>0$.
While we do not employ spectral methods directly,
the restriction $\theta < \frac{3}{4}$ is nevertheless essential for the analysis in the subsequent section
(see \eqref{eq:M}).
Specifically, it plays a crucial role in establishing
the required decay estimates via the energy method
(see the statement of Theorem \ref{T1}).
Regarding the right-hand inequality,
the asymptotic behavior of the function $\Omega$ (see Lemma \ref{L3bis1})
implies that this restriction is essential for ensuring that the remainder term $h$,
defined in \eqref{1.17}, belongs to $Y^1(\mathbb{R})$.

\begin{rem}
We summarize \textbf{Assumptions (I), (II), (III)}.
\begin{itemize}
\item \textbf{Assumption (I)}
means that the leading part of the linearized equation of \eqref{A} is given by
the corresponding parabolic equation \eqref{b0}.
\item \textbf{Assumption (II)} implies that the nonlinearity is subcritical and
the asymptotic profile should involve the nonlinear effect.
\item \textbf{Assumption (III)} is a technical condition
used to control remainder terms and to yield the exponential decay of the final energy.
\end{itemize}
\end{rem}

Finally, for all $s>s_0$, 
it will be convenient to also 
 introduce the quadratic form
\begin{align}\label{defphi}
\Phi (s,f(s),g(s)) &= \int_{\R}(1+y^{ 2})
(f^2+f_y^{2}) \y+ \int_{\R}
f_{yy}^{2} \y \\
\nonumber
&\quad
+\frac{e^{-s}}{a} \int_{\R}(y^{ 2} f_{yy}^{2} +f_{yyy}^{2}) \y 
+\frac{r^2e^{-s}}{a} \int_{\R}(1+y^{ 2})g^{2}  \y \\
&\quad
+\frac{r^2e^{-s}}{a} \int_{\R}g_{y}^{2}  \y.
\end{align}

We are now ready to state our main theorem in this paper.
\begin{theo}\label{T1}
Under \textbf{Assumptions (I), (II), (III)}, there exist $s_{0}>0$ and $\delta_{0}>0$,
such that,  for all $
(f_{0},g_{0}) \in Z^1(\R) $ with $
\|(f_0,g_0)\|_{Z^{1}(\R)} \le{\delta}_{0} $,
the equation (\ref{scaling}) has a unique solution
$(f,g)\in 
  {\mathcal C}([s_0,+\infty),Z^{1}(\R) )$
satisfying $ (f(s_{0}),g(s_{0}))=(f_{0},g_{0}) $.
Moreover, there exist $ \mu_{0}>0 $
and $C>0$, such that, for all $s\ge s_0$,
\begin{align}\label{P123}
	&\Phi (s,f(s),g(s))
	+ \int_{s_0}^s e^{-\mu_0(s-\tau)}
		\left( \|g(\tau)\|_{L^{2,1}}^2 +\|f_{yy}(\tau)\|_{L^2}^2 +\|g_y(\tau)\|_{L^2}^2 \right)
		\,\mathrm{d}\tau \\
	&\quad \le
	C \left( \Phi (s_0,f(s_0),g(s_0)) + e^{-\mu_0 s_0} \right) e^{-\mu_{0}(s-s_{0})}.
\end{align}
\end{theo}
\begin{theo}\label{T2}
Under \textbf{Assumptions (I), (II), (III)}, there exist
$t_{0}>0$ and $\delta_{0}>0$
such that, for the initial data
$(u_0, u_1)$ given at $t=t_0$ satisfying 
\[
	\|(u_0 - \Gamma(t_0) , u_1 - \partial_t \Gamma (t_0))\|_{X^{1}(\R)\times Y^{1}(\R)} \le{\delta}_{0},
\]
the solution $u$ of the equation \eqref{A} 
with the initial data
$(u(t_0), \partial_t u(t_0)) = (u_0, u_1)$
satisfies
\[
	u - \Gamma \in {\mathcal C}^0([t_{0},+\infty ), X^1(\R)) \cap  {\mathcal  C}^1([t_{0}, +\infty ), Y^1(\R)).
\]
Moreover, there exist
$ \mu_{0}>0$
and satisfies in particular the following estimate:
\begin{equation}\label{1.23}
	\| u(t) - \Gamma (t)\|_{L^{\infty}}
	={\mathcal O}(R(t)^{-\theta -\frac{\mu_{0}}{2}}), \qquad \qquad t \to +\infty .
\end{equation}
\end{theo}

Here, we give comparisons with previous studies, some remarks, and the strategy of the proof for
the above theorems.

The method of scaling variables coupled with energy estimates and various weighted energy estimates
is successfully used by Gallay and Raugel in \cite{GR1}.
They study the nonlinear damped wave equation in the one-dimensional space
\[
	\partial_{tt}u + \partial_t u = \partial_x \left( a(x) \partial_x u \right) + N\left(u, \partial_x u, \partial_t u \right),
\]
and prove that the solution has the same large-time behavior as
the corresponding parabolic equation for the supercritical nonlinearity.
In the subcritical case with the specific nonlinearity $-|u|^{p-1}u$,
in \cite{Hamza2}, the first author proves
an asymptotic stability result
near a one-parameter family
of the self-similar solutions of the associated semilinear parabolic equation.
Therefore, Theorems \ref{T1} and \ref{T2} can be regarded as variants of the result of \cite{Hamza2}
for the beam equation.

The study of the asymptotic behavior of solutions for the damped wave equation
with time-dependent coefficient
$\partial_{tt} u + b(t) \partial_t u - \partial_{xx}u=0$
goes back to Wirth in \cite{wi1} and \cite{wi2}.
He gives an interesting observation:
For $b(t) \sim (1+t)^{\beta}$, the asymptotic profiles of the solution are classified as
wave-type if $\beta < -1$ and heat-type if $-1<\beta <1$.
This observation can be generalized to the damped beam equation with two variable coefficients:
$\partial_{tt} u + b(t) \partial_t u - a(t) \partial_{xx}u + \partial_{xxxx} u=0$
with $a(t) \sim (1+t)^{\alpha}$ and $b(t)\sim (1+t)^{\beta}$.
It is known that in the constant coefficient case $a(t)=b(t) = 1$,
the solution behaves like that for the heat equation (see \cite{ta-yo}).
The classification of the behavior of solutions for general $(\alpha, \beta)$ can be found in \cite{YW1}.
Moreover, the corresponding nonlinear problem
$\partial_{tt} u + b(t) \partial_t u - a(t) \partial_{xx}u + \partial_{xxxx} u= \partial_x f(\partial_x u)$ 
is studied in the authors' previous result \cite{our1} in the supercritical case.
Our main theorems are counterparts to \cite{our1} for the subcritical case.
As far as the authors know, our theorems are new even in the constant coefficients case $a(t) = b(t)=1$.

\begin{rem}\label{rem:spec}
{(Heuristic dynamics and spectral properties)}
As detailed previously, the function $K(t,x)$ serves as an asymptotic solution to the linear problem and remains valid for the nonlinear problem \eqref{A123L} in the supercritical case.
In this setting, the nonlinear term acts as a perturbation that becomes asymptotically negligible as
$t \to +\infty$.
In other words, when the problem is reformulated in terms of the forward similarity variables $(y, s)$, the long-time asymptotic dynamics of the solution to \eqref{scaling} as $s \to +\infty$ are formally dictated by the following linear equation:
\begin{equation}
\label{w1.18}
f_{s} = \widetilde{\mathcal{L}}_{\infty}(f),
\end{equation}
where
\begin{equation*}
\widetilde{\mathcal{L}}_{\infty}(f) = f_{yy} + \frac{1}{2}yf_y + \frac{1}{2}f.
\end{equation*}
According to \cite[Appendix A]{GR1}, the spectrum of the differential operator $\widetilde{\mathcal{L}_{\infty}}$ in the function space $Y^1(\mathbb{R})$ satisfies 
\begin{equation*}
\sigma(\widetilde{\mathcal{L}}_{\infty}) = \left\{ -\frac{n}{2} \;\middle|\; n \in \mathbb{N} \right\} \cup \left\{ z \in \mathbb{C} \;\middle|\; \text{Re}(z) \le -\frac{1}{4} \right\} \cup \{0\}.
\end{equation*}
The spectrum of this  operator is contained in the half-plane $
\left\{ z \in \mathbb{C} \;\middle|\; \text{Re}(z) \le 0 \right\},$
 with the exception of the isolated principal eigenvalue $\lambda_0 = 0$, whose corresponding eigenfunction is given by the Gaussian profile $\varphi_0(y) = \frac{1}{\sqrt{4\pi}}e^{-y^2/4}$. Although these precise spectral estimates are not explicitly utilized in the proof, they are nevertheless essential for understanding the origin of the asymptotic expansion. Consequently, the long-time dynamics of the profile are driven primarily by these leading spectral elements. In fact, the aforementioned properties allow us to derive sharp, self-contained bounds by hand, utilizing refined energy estimates rather than relying on abstract spectral projections.

Obviously, in the subcritical regime, namely when $p < 1 + \frac{2(1-\beta)}{\alpha-\beta+1}$, the situation becomes significantly more challenging. Indeed, the nonlinear term is non-negligible at infinity and plays a decisive role in determining the precise asymptotic behavior of the solutions. More precisely, in this case, the dynamics are governed by a family of self-similar solutions to the associated parabolic problem. Formally, the behavior near this family of profiles is driven by the following equation:
\begin{equation}
\label{1.18}
f_{s} = \mathcal{L}_{\infty}(f),
\end{equation}
where
\begin{equation}
\label{1.18-bis}
\mathcal{L}_{\infty}(f) = \widetilde{\mathcal{L}}_{\infty}(f) + \left(\theta - \frac{1}{2}\right)f - p\Omega^{p-1}(y)f,
\end{equation}
and $\Omega$ is a solution of \eqref{1.4}.

We remark that the spectrum of the unperturbed operator ${\mathcal{L}}_{\infty}$ in the space $Y^1(\mathbb{R})$ is not explicitly known. However, since the profile $\Omega(y)$ decays sufficiently fast at infinity, the potential term $ - p\Omega^{p-1}(y)$ constitutes a localized, compact perturbation.
 Consequently, the essential spectrum of the  operator $\mathcal{L}_{\infty}$ is entirely determined by that of $\widetilde{\mathcal{L}}_{\infty}$ shifted by the real parameter $\theta - \frac{1}{2}$. By using the fact  that $\sigma_{\text{ess}}(\widetilde{\mathcal{L}}_{\infty})$ is contained in the half-plane $
\left\{ z \in \mathbb{C} \;\middle|\; \text{Re}(z) \le -\frac14\ \right\}$, it follows that $\sigma_{\text{ess}}(\mathcal{L}_{\infty})$ is contained in
$
\left\{ z \in \mathbb{C} \;\middle|\; \text{Re}(z) \le \theta-\frac34\ \right\}$.
Finally, \textbf{Assumption (III)} ensures that 
$\theta < \frac{3}{4}$,
which implies that the essential spectrum of $\mathcal{L}_{\infty}$ lies strictly in the left half-plane: 
$\sigma({\mathcal{L}_{\infty}}) \subset
\left\{ z \in \mathbb{C} \;\middle|\; \text{Re}(z) \le -\tilde \lambda \right\}$,
for some
$\tilde \lambda>0.$
Furthermore, the discrete spectrum will naturally shift under perturbation. To gain precise control over these eigenvalues, we exploit the relationship between the two operators, viewing  $\mathcal{L}_{\infty}$ as a perturbation of 
$\widetilde{\mathcal{L}}_{\infty}$. As seen in \eqref{1.18-bis}, they differ by exactly two terms:
\begin{itemize}
\item[{a)}] The first term, $\left(\theta - \frac{1}{2}\right)f$, induces a simple  translation of the spectrum to the right by $\theta - \frac{1}{2}$ (since $\theta > \frac{1}{2}$) compared to $\sigma(
\widetilde{
\mathcal{L}}_{\infty})$.
\item[{b)}] The second term, $-p\Omega^{p-1}(y)f$, acts as a localized negative potential, which formally pulls the spectrum to the left.
\end{itemize}
Although the competing nature of these two algebraic shifts prevents an immediate localization of the discrete spectrum, this difficulty can be bypassed by means of sharp energy estimates. More precisely, by implementing a careful energy argument, we establish that the superposition of these two competing effects forces the entire spectrum of the operator ${\mathcal{L}_{\infty}}$ in $Y^1(\mathbb{R})$ to be strictly confined to the left complex half-plane. That is, under
\textbf{Assumptions (I), (II), (III)}, there exists a constant $\tilde{\lambda} > 0$ such that
\begin{equation}\label{3543}
\sigma({\mathcal{L}_{\infty}}) \subset
\left\{ z \in \mathbb{C} \;\middle|\; \text{Re}(z) \le -\tilde \lambda\ \right\}.
\end{equation}

To clarify the role of these spectral bounds, let us outline the strategy for deriving the refined energy estimates that form the technical core of this work. A standard approach relying on classical energy functionals equipped with the typical polynomial weight $(1+y^2)$, as employed in \cite{GR1, our1}, fails in the present context. Indeed, a direct calculation of the quadratic form
\begin{equation}
 \label{bad}
 \text{Re} \left( \int_{\mathbb{R}} \mathcal{L}_{\infty}(f) \bar{f}  (1+y^2) \mathrm{d}y \right)
 \end{equation}
 generates problematic positive terms arising from the tail behavior of the subcritical profile $\Omega(y)$ and the cross-terms of the operator. These coherent structures cannot be absorbed by the dissipative terms of the equation, preventing one from establishing the localized coercivity needed for decay.
To overcome this fundamental structural obstruction, we introduce a novel corrective weight $q(y)$, constructed explicitly by using the subcritical profile $\Omega(y)$ (as detailed below in \eqref{3.28}), which remains controlled by the standard polynomial weight $(1+y^2)$. From an analytical standpoint, this modified framework allows us to exploit crucial algebraic cancellations inherent to the linearized operator, yielding the fundamental estimate:
\begin{equation} \label{cc}
\text{Re} \left( \int_{\mathbb{R}} \mathcal{L}_{\infty}(f) \bar{f}  \mathrm{d}\mu \right) \le -\tilde{\lambda} \int_{\mathbb{R}} |f|^2 \mathrm{d}\mu,
\end{equation}
where $\mathrm{d}\mu(y) = \left(1 + y^2 + \tilde{c} q(y)\right) \,\mathrm{d}y$ denotes the modified weighted measure, for some $\tilde{c} > 0$, and $\tilde{\lambda} > 0$. Crucially, inequality \eqref{cc} implies \eqref{3543}.
 On the level of the full coupled system \eqref{scaling}, this mechanism is captured by introducing a dedicated corrective functional
 based on a special weight function $q(y)$
 (for the definition, see \eqref{3.2} and \eqref{3.28}).
 This functional is tailored so that it is equivalent to the standard energy
 and enables us to obtain the exponential decay (see Section 2.1).
\end{rem}

The outline of this paper is as follows.
In the next section,
we give the a priori estimates of the solution in terms of weighted energy.
In Appendices, we collect some technical lemmas used in the energy estimates,
and we prove the local existence of the unique mild solution belonging to weighted energy spaces.



\section{Proofs of main theorems}\label{sec2}
The local existence of mild solutions in weighted Sobolev spaces
with initial data given at an arbitrary initial time $s=s_0$
is proved in Appendix B.
Therefore, it suffices to prove the a priori estimate which guarantees
the global existence and the decay of solutions.
To this end, throughout this section, we assume that for $s_{0}>0$ and $S>0$,
we are given a solution
$ (f,g)\in {\mathcal C}^{0}([s_{0},s_{0}+S], Z^1(\R))$ of (\ref{scaling}) with
initial data $(f_{0},g_{0}) \in Z^1(\R)$   such that $\|(f_0,g_0)\|_{Z^1(\R)}
\le \delta_0$, which satisfies
\begin{equation}\label{3.1}
	\|f(s)\|^2_{H^{1,1}(\R)}+\|f_{yy}(s)\|^2_{L^2(\R)}
	\le\kappa^2 \delta^2_{0}<1,\qquad \
	s\in [s_{0},s_{0}+S],
\end{equation}
where $\kappa$ is a real number that will be fixed later
and $\delta_{0}$ is small enough
such that $\kappa\delta_{0}<1$.
In order to apply general energy identity in Lemma \ref{LG},
we sometimes use the following rearranged version of the system \eqref{scaling}:
\begin{equation}\label{scalingh}
\left\{
\begin{alignedat}{3}
&g=f_{s}-\frac{y}{2}f_y-\theta  f,\\
&\frac{r^2e^{-s}}{a}\Big(g_s-\frac{y}2g_y-(\theta  +1) g\Big)+\frac{r'}{a}g
+g=-\frac{e^{-s}}{a}f_{yyyy}+f_{yy}+k+h,
\end{alignedat}
\right.
\end{equation}
where $h$ is  given by \eqref{1.17}, and  
\begin{equation}\label{defk}
k=-\big(|f+\ \f |^{p-1}(f+\ \f )
-\ \f ^{p}(y)\big).
\end{equation}
Throughout this section, the symbol $C$ stands for a generic constant
that may change from line to line.

\subsection{Part I: Weighted $L^2$ energies}
We first introduce the following energy functionals.
\begin{equation}
\label{3.2}
\begin{alignedat}{3}
&E_2^{(0,1)}(s) = \frac{1}{2} \int_{\R}(1+y^2) f^{2}  \,{\mathrm{d}}y
	+ \frac{r^2e^{-s}}{a} \int_{\R} (1+y^2) f g \, {\mathrm{d}}y, \\
&E_{q}(s)=\frac{1}{2}\displaystyle\int_{\R} q(y)f^{2} \,{\mathrm{d}}y
	+\frac{r^2e^{-s}}{a} \int_{\R} q(y) f g \,{\mathrm{d}}y, \\
& \E_{\rho}(s)=E_{2}^{(0,1)}(s)+\rho E_{q}(s),
\end{alignedat}
\end{equation}
where the function $q(y)$ is defined by \eqref{3.28} and its properties are summarized in Lemma \ref{L3},
and $\rho > 0$ is a sufficiently large constant determined later.
The above energy functionals correspond to the weighted $L^2$-norm of $f$.

\begin{lem}\label{L1}
$E_{2}^{(0,1)} \in{\mathcal C}^1([s_{0},s_{0}+S])$ and, for
all $s\in [s_{0},s_{0}+S]$,
\begin{align}\label{3.3}
	\frac{d}{ds}E_{2}^{(0,1)}(s)
	&=
	-\frac{e^{-s}}{a}\int_{\R}(1+y^{2})f_{yy}^2 \,{\mathrm{d}}y
	-\int_{\R}(1+y^{2}) f_{y}^2 \,{\mathrm{d}}y
	+ \left( \theta -\frac{ 3}{4} \right) \int_{\R}y^2f^2 \,{\mathrm{d}}y \\
\nonumber
	&\quad
	+ \left( \theta  +\frac{3}{4} \right) \int_{\R}f^2{\mathrm{d}}y
	+ \Lambda_{2}^{(0,1)} (s),
\end{align}
where $\Lambda_{2}^{(0,1)}$ satisfies
\begin{equation}\label{3.300A}
	\Lambda_{2}^{(0,1)}
	\le
	C_0 (\kappa\delta_{0})^{\overline {p}-1} \int_{\R}(1+y^{2})f^2 \,{\mathrm{d}}y
	+Ce^{-\mu s} \int_{\R}(1+y^{2})g^2 \,{\mathrm{d}}y +Ce^{-\mu s},
\end{equation}
with $\overline{p}= \min (2,p)$.
\end{lem}
\begin{proof}
For $n= 0,1$, let
\[
	E_2^{(n)}(s) := \frac{1}{2} \int_{\mathbb{R}} y^{2n}
		\left( f^2 + \frac{r^2e^{-s}}{a} fg \right) \, \mathrm{d}y.
\]
Then, applying Lemma \ref{LG} with
$c_1(s) = \dfrac{r^2e^{-s}}{a}$,
$c_2(s) = \dfrac{r'}{a}$,
$c_3(s) = \dfrac{e^{-s}}{a}$
to the system \eqref{scalingh}
with noting \eqref{eq:c1prime}, we obtain
\begin{equation}\label{3.7}
	\frac{d}{ds}{E_{2}^{(0)}}(s)=-\frac{e^{-s}}{a}\int_{\R}f^2_{yy} \,{\mathrm{d}}y
	-\int_{\R} f_{y}^2 \,{\mathrm{d}}y
	+ \left( \theta  -\frac{1}{4} \right) \int_{\R}f^2 \,{\mathrm{d}}y
	+ \Lambda _{2}^{(0)} (s),
\end{equation}
where
\begin{align}\label{3.7.bis}
	\Lambda _{2}^{(0)} (s)
	&=
	-p\int_{\R}\ \f ^{p-1}f^2 \,{\mathrm{d}}y+\int_{\R}f{\mathcal N}(f) \,{\mathrm{d}}y \\
\nonumber
	&\quad +\frac{r^2e^{-s}}{a}
		\int_{\R}\Big(g^2+(2\theta  -\frac{1}{2})fg\Big) \,{\mathrm{d}}y
		-\frac{b'}{b^2}\int_{\R}fg \,{\mathrm{d}}y
		+\int_{\R}hf \,{\mathrm{d}}y,
\end{align}
In the same way, we have
\begin{align}\label{ss0}
	\frac{d}{ds}{E_{2}^{(1)}}(s)
	&=
	-\frac{e^{-s}}{a}\int_{\R}y^{2}f_{yy}^2 \,{\mathrm{d}}y
	-\int_{\R}y^{2} f_{y}^2 \,{\mathrm{d}}y
	+ \left( \theta -\frac{3}{4} \right) \int_{\R}y^{2}f^2 \,{\mathrm{d}}y
	+\int_{\R}f^2 \,{\mathrm{d}}y
	+ \Lambda_{2}^{(1)} (s),
\end{align}
where
\begin{align}\label{ss1}
	\Lambda_{2}^{(1)} (s)
	&=
	-p\int_{\R} \f ^{p-1}y^{2}f^2{\mathrm{d}}y+\int_{\R}y^2f{\mathcal N}(f) \,{\mathrm{d}}y
	+4 \frac{e^{-s}}{a}\int_{\R}f^2_{y} \,{\mathrm{d}}y
	\nonumber\\ 
	&\quad
	+\frac{r^2e^{-s}}{a} \int_{\R}y^2 \Big(g^2+(2\theta -\frac{ 3}{2})fg\Big) \,{\mathrm{d}}y
	-\frac{b'}{b^2}\int_{\R}y^{2}fg \, {\mathrm{d}}y
	+\int_{\R}y^{2}hf \,{\mathrm{d}}y.
\end{align}
Adding the above two identities, we deduce \eqref{3.3} with
\begin{align}\label{3.7.}
	\Lambda_{2}^{(0,1)} (s)
	&:= \Lambda_{2}^{(0)} (s) + \Lambda_{2}^{(1)}(s) \\
	&=\underbrace{ \frac{r^2e^{-s}}{a} \int_{\R}\Big(g^2+(2\theta -\frac{1}{2})fg\Big) \,{\mathrm{d}}y}_{I_1(s)}
	+\underbrace{\frac{r^2e^{-s}}{a} \int_{\R}y^2\Big(g^2+(2\theta -\frac{ 3}{2})fg\Big) \,{\mathrm{d}}y}_{I_2(s)}
	\nonumber\\
	&\quad +\underbrace{4 \frac{e^{-s}}{a}\int_{\R}f^2_{y} \,{\mathrm{d}}y}_{I_3(s)}
		\underbrace{-\frac{b'}{b^2}\int_{\R}(1+y^{2})fg \,{\mathrm{d}}y}_{I_4(s)}
		+\underbrace{\int_{\R}(1+y^{2})hf \,{\mathrm{d}}y}_{I_5(s)}
\nonumber\\
	&\quad + \underbrace{\int_{\R}(1+y^{2}) f \mathcal{N}(f) \,{\mathrm{d}}y}_{I_6(s)}
		-p\int_{\R} \Omega^{p-1}(1+y^{2})f^2 \,{\mathrm{d}}y.
\nonumber
\end{align}
It remains then to prove the estimate \eqref{3.300A}.
Combining the Schwarz inequality with \eqref{v1} and \eqref{v2} in Lemma \ref{lem:remainder},
we easily obtain
\[
	\sum_{i=1}^4|I_i(s)| \le Ce^{- \mu s}\Big( \|f(s)\|^2_{H^{1,1}(\R)}+ \|g(s)\|^2_{H^{0.1}(\R)}\Big).
\]
Thanks to \eqref{3.1}, we have
\[
	\sum_{i=1}^4|I_i(s)| \le Ce^{- \mu s} \|g(s)\|^2_{H^{0.1}(\R)}+Ce^{- \mu s}.
\]
For the term $I_5$, applying \eqref{3.1} and Lemma \ref{lem:est:h}, we conclude
\begin{equation}\label{65231}
	|I_5(s)| \le  C\|f(s)\|_{H^{0.1}(\R)} \|h(s)\|_{H^{0.1}(\R)}\le C e^{- \mu s}.
\end{equation}
Moreover, we handle the term $I_6(s)$ thanks to Lemma \ref{LN} and obtain
\[
	|I_6(s)| \le C (\kappa \delta_0)^{\bar{p}-1} \int_{\mathbb{R}} (1+y^2) f^2 \,\mathrm{d}y.
\]
Finally, since the function $\f (y)$ is positive, the last term in the right-hand side of \eqref{3.7.} is non-positive.
Putting the above estimates together, we have the estimate \eqref{3.300A}.
\end{proof}

The inequality (\ref{3.3}) still includes a positive term
\[
	\left( \theta  +\frac{3}{4} \right) \int_{\R}f^{2} \,{\mathrm{d}}y,
\]
which is needed to be controlled.
To surmount this, we introduce a new energy functional
$E_q(s)$, which includes a weight function $q(y)$ defined by \eqref{3.28}.

Let us state now the following  auxiliary result:
\begin{lem}\label{L2}
Let $E_q(s)$ be defined in \eqref{3.2}.
Then,
$E_{q}(s) \in {\mathcal C}^1([s_{0},s_{0}+S])$
and for all $s\in [s_{0},s_{0}+S]$,
\begin{equation}\label{3.15}
	\frac{d}{ds}{E_{q}}(s)\le -(p-1)\int_{\R} \f ^{p-1}q(y)f^2{\mathrm{d}}y 
	+ \Lambda_{q}(s),
\end{equation}
where
\begin{equation}\label{3.15a}
	\Lambda _{q}(s) \le 
	C (\kappa\delta_{0})^{\overline {p}-1} \int_{\R}(1+y^2)f^2 \,{\mathrm{d}}y
	+ C e^{-\mu  s}\int_{\R}(1+y^2)g^2 \,{\mathrm{d}}y+ C e^{-\mu s}
\end{equation}
with $\bar{p}=\min(2, p)$.
\end{lem}
\begin{proof}
Using the equation \eqref{scaling}, we obtain
\begin{multline*}
	\frac{d}{ds}{E_{q}}(s)
	= -\frac{e^{-s}}{a}\int_{\R}q(y)ff_{yyyy} \,{\mathrm{d}}y
	+ \int_{\R}q(y)f{\mathcal L}(f) \,{\mathrm{d}}y
	+\int_{\R}q(y)f{\mathcal N}(f) \,{\mathrm{d}}y \\
	+\frac{r^2e^{-s}}{a}
		\int_{\R}q(y)\Big( \frac{y}2(fg)_y+g^2+2\theta fg \Big)\,{\mathrm{d}}y
	-\frac{b'}{b^2} \int_{\R}q(y)fg \,{\mathrm{d}}y
	+\int_{\R}q(y)hf \,{\mathrm{d}}y.
\end{multline*}
Then by integrating by parts, we conclude that 
\begin{equation}
\label{3.18}
\frac{d}{ds}{E_{q}}(s)=\int_{\R} q(y)f{\mathcal L}(f)
{\mathrm{d}}y
+\Lambda _{q}(s),
\end{equation}
where
\begin{align}
\label{3.182}
\Lambda _{q}(s)=&-\frac{e^{-s}}{a}\int_{\R}q(y)f^2_{yy}
{\mathrm{d}}y-\frac{e^{-s}}{a}\int_{\R}q''(y)ff_{yy}
{\mathrm{d}}y+\frac{e^{-s}}{a}\int_{\R}q''(y)f_y^2
{\mathrm{d}}y\\
&+\int_{\R}q(y)f{\mathcal
  N}(f){\mathrm{d}}y  +\frac{r^2e^{-s}}{a}
\int_{\R}q(y)\Big(g^2+(2\theta 
-\frac{1}{2})fg\Big){\mathrm{d}}y \nonumber\\
 &- \frac{r^2e^{-s}}{2a}
\int_{\R}yq'(y)fg{\mathrm{d}}y
-\frac{b'}{b^2}\int_{\R}q(y)fg{\mathrm{d}}y+\int_{\R}q(y)hf{\mathrm{d}}y.\nonumber
\end{align}
Recalling the definition of ${\mathcal L}$ given by \eqref{1.17}, and integrating by parts, we get
\begin{align}\label{3.19}
	\int_{\R}q(y)f{\mathcal L}(f){\mathrm{d}}y
	&=
	-\int_{\R}q(y)f_{y}^2 \,{\mathrm{d}}y
	+\frac{1}{2} \int_{\R}\ \Big(q''(y)-\frac{y}{2}q'(y) -\frac{1}{2}q(y)\Big)f^2 \,{\mathrm{d}}y\nonumber\\
	&\quad
	+\int_{\R}\ (\theta - \f ^{p-1}(y))q(y)f^2\,{\mathrm{d}}y
	-(p-1)\int_{\R}\f ^{p-1}(y)q(y)f^2 \,{\mathrm{d}}y.
\end{align}
Now, observe that by (\ref{1.4}), we write
$\theta - \f ^{p-1}(y)=-\frac{\f''(y) +\frac{y}{2} \f'(y)} {\ \f (y)}.$
Consequently, by  using  (\ref{3.19}), we deduce that
\begin{align}\label{3.21}
	\int_{\R}q(y) f{\mathcal L}(f){\mathrm{d}}y
	&=
	\frac{1}{2}\int_{\R}\Big(q''(y)-\frac{y}{2}q'(y)-(\frac{y\f'(y)}{\ \f (y)}+\frac{1}{2})q(y)\Big)f^2 \,{\mathrm{d}}y
	-\int_{\R}q(y)f_{y}^2 \,{\mathrm{d}}y\\
\nonumber
	& - \int_{\R}\frac{ \f'' (y)}{ \f (y)}q(y)f^2 \,{\mathrm{d}}y
	-(p-1)\int_{\R}\f ^{p-1}(y)q(y)f^2 \,{\mathrm{d}}y.
\end{align}
Integrating by parts again implies
\begin{align}\label{3.22}
	\int_{\R}q(y) f{\mathcal L}(f){\mathrm{d}}y
	&=
	\frac{1}{2}\int_{\R}\Big(
		q''(y) + \left( \frac{2 \f' (y)}{\f (y)}-\frac{y}{2} \right) q'(y)
				- \left( \frac{y \f' (y)}{ \f } +\frac{1}{2} \right) q(y)
		\Big)f^2 \,{\mathrm{d}}y\nonumber\\
	&\quad
	-\int_{\R}q(y) f_{y}^2 \,{\mathrm{d}}y
	-\int_{\R} \frac{{\f'}^{2}(y)}{\f^{2}(y)} q(y)f^2 \,{\mathrm{d}}y
	+2\int_{\R}\ \frac{ \f' (y)}{ \f (y)}q(y)ff_y \,{\mathrm{d}}y \\
	&\quad
	-(p-1)\int_{\R} \f ^{p-1}(y) q(y)f^2 \,{\mathrm{d}}y.\nonumber
\end{align}
Thanks to the Schwarz inequality, we infer
\begin{equation}\label{4july}
 -\int_{\R}q(y)
f_{y}^2{\mathrm{d}}y
-\int_{\R} 
\frac{{\f'}^{2}(y)}{\f^{2}(y)}
q(y)f^2{\mathrm{d}}y
+2\int_{\R}\frac{ \f' (y)}
{\ \f (y)}q(y)ff_y
{\mathrm{d}}y\le 0.
\end{equation}
Clearly  (\ref{3.22}), \eqref{4july},   and (\ref{3.25}), gives
\begin{equation}
\label{3.22.}
\int_{\R}q(y) f{\mathcal L}(f){\mathrm{d}}y\le
-(p-1)\int_{\R}\f ^{p-1}q(y)f^2
{\mathrm{d}}y.
\end{equation}
Combining (\ref{3.18}) and (\ref{3.22.}),
we obtain the inequality
(\ref{3.15}), where
$\Lambda _{q}(s)$
is given by 
\eqref{3.182}.
To control 
$\Lambda_{q}(s)$,
we follow a similar way utilised  in the  bound of
$\Lambda_{2}^{(0,1)}(s)$.
Obviously, we use also the 
estimates (\ref{3.1}), (\ref{3.34}), \eqref{3.35}, and  \eqref{3.35bis}  useful to control the  functions  $q, yq'$ and $q''$. Hence,  we
 obtain the estimate  (\ref{3.15a}).
 This concludes the proof of Lemma \ref{L2}.
\end{proof}

Finally, combining the above two lemmas and choosing the constant $\rho$
suitably large, we have the following good estimate for $\mathbb{E}_{\rho}(s)$.
\begin{lem}\label{L03}
Let $\E_{\rho}(s)$ be defined in \eqref{3.2}.
$\E_{\rho} \in{\mathcal C}^1([s_{0},s_{0}+S])$
and there exist
$C_1>0$ 
and  $\rho_0>0$,
such that, for  all $s\in [s_{0},s_{0}+S]$,
\begin{align*}
	\frac{d}{ds} {\E}_{\rho_0}(s)
	&\le 
	- \frac{1}{2} \left( \frac{ 3}{4}-\theta \right) \int_{\R}(1+y^2)f^{2} \,{\mathrm{d}}y
	-\int_{\R}(1+y^2)f_{y}^{2} \,{\mathrm{d}}y \\
	&\quad
	- \frac{e^{-s}}{a}\int_{\R}(1+y^2)f_{yy}^2 \,{\mathrm{d}}y
	+\Lambda _{\rho_0}(s),
\end{align*}
where $\Lambda _{\rho_0}$ satisfies,
with some constant $C_{\ast}>0$,
\begin{equation}\label{3.300}
	\Lambda _{\rho_0}(s)
	\le  C_{\ast} (\kappa \delta_0)^{\bar p -1} \int_{\R}(1+y^{2})f^2 \,{\mathrm{d}}y
	+ C_{\ast} e^{-\mu s} \int_{\R} (1+y^{2})g^2 \, {\mathrm{d}}y
	+ C_{\ast} e^{-\mu s}.
\end{equation}
\end{lem}
\begin{proof}
Clearly, for all $\rho>0$,
$\E_{\rho} \in{\mathcal C}^1([s_{0},s_{0}+S])$
and
$\dfrac{d}{ds} {\E}_{\rho}(s)=\dfrac{d}{ds} {E_{2}^{(0,1)}}(s)+\rho \dfrac{d}{ds}  {E_{q}}(s)$.
Using (\ref{3.3}) and (\ref{3.15}), we get for all $\rho>0$,
\begin{align}\label{3.42}
	\frac{d}{ds} {\E}_{\rho}(s)
	&\le
	-\frac{e^{-s}}{a}\int_{\R}(1+y^2)f_{yy}^2 \,{\mathrm{d}}y
	-\int_{\R}(1+y^2)f_{y}^{2} \,{\mathrm{d}}y
	+\Lambda_{2}^{(0,1)}(s) + \rho \Lambda_{q}(s)
	\nonumber\\
	&\quad
	+ \left( \theta  -\frac{ 3}{4} \right) \int_{\R}y^2f^{2} \,{\mathrm{d}}y 
	+ \left( \theta +\frac{3}{4} \right) \int_{\R}f^{2} \,{\mathrm{d}}y \\
	&\quad
	-\rho (p-1)\int_{\R} \f ^{p-1}q(x)f^2 \,{\mathrm{d}}y.
\end{align}
Note that, one can neglect the first term of the right-hand side in \eqref{3.42}.
Using the fact that $\theta < 3/4$, which follows from \textbf{Assumption (III)},
we see that the coefficient of the fifth term is negative.
Moreover, the terms $\Lambda _{2}^{(0,1)}(s)$ and $\rho \Lambda _{q}(s)$ are already estimated in
Lemmas \ref{L1} and \ref{L2}.
Therefore, we focus on  the control of the bad term
$\displaystyle \left( \theta +\frac{3}{4} \right) \int_{\R}f^{2} \,{\mathrm{d}}y$.
It is easy to see that
\[
	\left( \theta + \frac{3}{4} \right) \int_{|y| > M} f^2 \,\mathrm{d}y
	+ \frac{1}{2} \left( \theta  -\frac{ 3}{4} \right) \int_{|y| > M} y^2f^{2} \,{\mathrm{d}}y \le 0
\]
if $M$ is chosen so that
\begin{equation}\label{eq:M}
	\left( \theta + \frac{3}{4} \right) \dfrac{1}{M^2} + \frac{1}{2} \left( \theta  -\frac{ 3}{4} \right) \le 0.
\end{equation}
Since $\f (x)$ is even and  non increasing function in $(0,\infty)$,
we infer
\begin{equation*}
	\int_{|y|<M} \f ^{p-1}(x)q(x)f^2 \,{\mathrm{d}}y
	\ge
	\tilde C \f ^{p-1}(M) \int_{|y|<M}f^{2} \,{\mathrm{d}}y,
\end{equation*}
where
$ \tilde C=\inf_{x\in \R}q(x)$. Now we choose
 $\rho=\rho_0$ large enough so that
\begin{equation*}
	\theta +\frac{3}{4} -\rho_0(p-1)\tilde C \f ^{p-1}(M_{0}) \le - \frac{1}{2} \left( \frac{ 3}{4}-\theta \right). 
\end{equation*}
Consequently, the last three terms of \eqref{3.42} can be estimated by
\[
	- \frac{1}{2} \left( \frac{ 3}{4}-\theta \right)  \int_{\R}(1+y^{2})f^2 \,{\mathrm{d}}y.
\]
Applying this to \eqref{3.42} concludes the proof of Lemma \ref{L03}.
\end{proof}
%
%
%
%
%
%
%

\subsection{Part II: Weighted first order energies}
Since we want to control  the norm of $(f(s),g(s))$ in $Z^{1}(\R) $, 
 it is natural to introduce the following energy
functional:
\begin{equation}\label{3.51b}
	E_{1}^{(0,1)}(s)
	= \frac{1}{2} \int_{\R} (1+y^2) f_y^{2} \,\y
	+ \frac{e^{-s}}{2a} \int_{\R} (1+y^2) f_{yy}^{2} \,\y
	+\frac{r^2e^{-s}}{2a} \int_{\R} (1+y^2) g^{2} \,\y.
\end{equation}
\begin{lem}\label{L5}
Let $E_1^{(0,1)} (s)$ be defined in \eqref{3.51b}.
Then $E_{1}^{(0,1)} \in{\mathcal C}^1([s_{0},s_{0}+S])$,
and for all $s\in [s_{0},s_{0}+S]$,
\begin{align}
\label{3.52b}
	\frac{d}{ds}  {E_{1}^{(0,1)}}(s)
	&\le
	-\frac{1}{2}\int_{\R} (1+y^{ 2})g^{2} \,{\mathrm{d}}y
	- \frac{4 e^{-s}}{a} \int_{\R}yf_{yy} g_y \,\y
	+ \Lambda_1^{(0,1)}(s),
\end{align}
where $\Lambda_1^{(0,1)}(s)$ satisfies, with some constant $C_{\ast\ast}>0$,
\[
	\Lambda_1^{(0,1)}(s) \le
	C_{\ast\ast} \int_{\R}(1+y^{ 2})\big(f^{2}+f^2_{y}\big) \,{\mathrm{d}}y
	+ C_{\ast \ast} \frac{e^{-s}}{a} \int_{\R}(1+y^2)f^2_{yy} \,\y
	+ C_{\ast\ast} e^{-\mu s}.
\]
\end{lem}
\begin{proof}
For $n=0, 1$, let
\[
	E_1^{(n)}(s) :=
	\frac{1}{2} \int_{\R} y^{2n} f_y^{2} \,\y
	+ \frac{e^{-s}}{2a} \int_{\R} y^{2n} f_{yy}^{2} \,\y
	+\frac{r^2e^{-s}}{2a} \int_{\R} y^{2n} g^{2} \,\y.
\]
By applying Lemma \ref{LG} with 
$c_1(s)=\dfrac{r^2e^{-s}}{a}$,
$c_2(s)=\dfrac{r'}{a}$,
$c_3(s)=\dfrac{e^{-s}}{a}$,
$l=\theta$, $m=\theta+1$,
and also using \eqref{eq:c1prime} and \eqref{eq:c3prime}, we have
\begin{multline}
	\frac{d}{ds}E_{1}^{(0)} (s)
	=
	-\int_{\R}g^2 \, \y + \left( \theta+\frac14 \right) \int_{\R}f_y^2 \,\y
	+ \left( \theta+\frac{1}{4} \right) \frac{e^{-s}}{a} \int_{\R}f^2_{yy} \,\y\nonumber\\
	+ \left( \theta+\frac{1}{4} \right) \frac{r^2e^{-s}}{a} \int_{\R}g^2 \,\y 
	-\frac{r a'}{2a^2} \int_{\R}g^2 \,\y
	-\frac{a'}{2ra^2} \int_{\R}f_{yy}^2 \,\y
	+ \int_{\R} kg \,\y + \int_{\R} hg \,\y.
\end{multline}
%
%
Using the Cauchy--Schwarz inequality, \eqref{v1},  \eqref{v2}, and  \eqref{3.1}, we show that
\begin{equation}
\label{3.55b}
\frac{d}{ds} E_{1}^{(0)}(t) \le
-\frac{1}{2}\int_{\R}g^2{\mathrm{d}}y
+C\int_{\R}\big(
f^2+f_{y}^2\big)
{\mathrm{d}}y+C\frac{e^{-s}}{a}
 \displaystyle\int_{\R}f^2_{yy}\y
+C\displaystyle\int_{\R}(k^2+h^2)\y.
\end{equation}
Likewise, $E_{1}^{(1)}$ is differentiable in $s \in
[s_{0},s_{0}+S]$ and
\begin{multline}
\frac{d}{ds}E_{1}^{(1)}(s)=-\displaystyle\int_{\R}y^2 g^2\y +
\left( \theta-\frac{1}4 \right) \displaystyle\int_{\R}y^2f_y^2\y
+ \left( \theta-\frac{1}4 \right) \frac{e^{-s}}{a}\displaystyle\int_{\R}y^2f^2_{yy}\y\nonumber\\
+\left( \theta-\frac{1}4 \right) \frac{r^2e^{-s}}{a}\displaystyle\int_{\R}y^2g^2
  \y -
\frac{a'r}{2a^2}
\displaystyle\int_{\R}y^2g^2\y
- \frac{4e^{-s}}{a}\displaystyle\int_{\R}yf_{yy}g_y\y\\
-\frac{2e^{-s}}{a}\displaystyle\int_{\R}f_{yy}
g\y
-2\int_{\R}y
f_{y}g\y
 -\frac{a'}{2a^2r}\displaystyle\int_{\R}y^2f_{yy}^2
\y  +
\int_{\R}
y^2kg\y+
\int_{\R}
y^2hg\y.
\end{multline}
Similarly, we conclude  that there
exists $C>0$
such that,
\begin{multline}
\label{3.58b}
\frac{d}{ds} E_{1}^{(1)}(s) \le
 -\frac{1}{2}\int_{\R}y^2g^{2}
{\mathrm{d}}y
-\frac{4 e^{-s}}{a}\displaystyle\int_{\R}yf_{yy}
g_y\y+C\int_{\R}(1+y^{ 2})\big(f^{2}+
f^2_{y}\big)
{\mathrm{d}}y\\
 +C\frac{e^{-s}}{a}\displaystyle\int_{\R}(1+y^2)f^2_{yy}\y
+C e^{-s}\int_{\R}
g^2{\mathrm{d}}y
+C\displaystyle\int_{\R}y^2(k^2+h^2)\y.
\end{multline}
Consequently, by choosing $s_0$ large enough,   (\ref{3.55b}),
(\ref{3.58b}), \eqref{h02fk}, and \eqref{h},   we get (\ref{3.52b}). 
 This concludes the proof of Lemma \ref{L5}.
\end{proof}

\subsection{Part III: Higher-order energies}
Finally,  to control the bad term $-\dfrac{4 e^{-s}}{a}\displaystyle\int_{\R}yf_{yy}
g_y\y,$ we define  the following functionals:
\begin{equation}
\label{3.2bisb}
\begin{alignedat}{3}
	&E_{2}^{\langle 1 \rangle, (0)} (s) &&=\frac{1}{2} \int_{\R}f_{y}^{2} \,{\mathrm{d}}y
	+ \frac{r^2e^{-s}}{a}\int_{\R}f_{y} g_{y} \,{\mathrm{d}}y,\\
	&E_{1}^{\langle 1 \rangle, (0)} (s) &&=\frac{1}{2} \int_{\R} f_{yy}^{2} \,\y
	+ \frac{e^{-s}}{2a} \int_{\R}f_{yyy}^{2} \,\y
	+ \frac{r^2e^{-s}}{2a} \int_{\R}g_{y}^{2} \,\y,\\
	&\E_{\vartheta, \zeta, \omega  }(s) &&= \mathbb{E}_{\rho_0}(s)
	+ \vartheta E_{1}^{(0,1)}(s) + \zeta E_{2}^{\langle 1 \rangle, (0)} (s) + \omega E_{1}^{\langle 1 \rangle, (0)} (s),
\end{alignedat}
\end{equation}
where $\vartheta, \zeta, \omega  >0$ will be chosen later.

Now, we shall observe that 
by differentiating \eqref{scaling}, we infer that
 $(f_y,g_y)$ satisfy the  following system: 
\begin{equation}\label{SG1}
	\left\{
	\begin{alignedat}{2}
		&g_y=(f_y)_{s}-\frac{y}{2}(f_y)_y- \left( \theta+\frac12 \right)f_y,\\
		&c_1(s)\Big((g_y)_s-\frac{y}2(g_y)_y- \left( \theta+\frac32 \right) g_y \Big)+c_2(s)g_y
			+g_y=-c_3(s)f_{yyyyy}+f_{yyy}+k_y+h_y.
	\end{alignedat}
	\right.
\end{equation}
By applying Lemma \ref{SG} with
$c_1(s)=\dfrac{r^2e^{-s}}{a}$, $c_2(s)=\dfrac{r'}{a}$, $c_3(s)=\dfrac{e^{-s}}{a}$,
$l=\theta+\frac12$ and $m=\theta+\frac32$,
and also using \eqref{eq:c1prime} and \eqref{eq:c3prime}, we have the following.
\begin{lem}\label{L6}
Let
$E_{2}^{\langle 1 \rangle, (0)}(s)$
be defined in \eqref{3.2bisb}.
Then,
$E_{2}^{\langle 1 \rangle, (0)} \in{\mathcal C}^1([s_{0},s_{0}+S])$
and there exists $C_{\sharp} > 0$ such that,
for all $s\in [s_{0},s_{0}+S]$,
\begin{align}\label{3.3b}
	\frac{d}{ds}E_{2}^{\langle 1\rangle, (0)}(s)
	&=-\frac12\int_{\R} f_{yy}^2 \,{\mathrm{d}}y
	-\frac{e^{-s}}{a}\int_{\R}f_{yyy}^2 \,{\mathrm{d}}y
	+ \Lambda_{2}^{\langle 1\rangle, (0)} (s),
\end{align}
where $\Lambda_{2}^{\langle 1\rangle, (0)}$ satisfies
\begin{align}\label{3.300Ab}
	\Lambda_{2}^{\langle 1\rangle, (0)}(s)
	&\le  C_{\sharp} \int_{\R}(1+y^2) f^2 \,{\mathrm{d}}y
		+C_{\sharp} \int_{\R}f_y^2 \, {\mathrm{d}}y 
		+C_{\sharp} e^{-\mu s} \int_{\R}g^2 \,{\mathrm{d}}y
		+C_{\sharp} e^{-\mu s}.
\end{align}
\end{lem}
\begin{proof}
$E_{2}^{\langle 1\rangle, (0)}$ is
a differentiable function for  $ s \in [s_0,s_0+S]$
and  as explained above, by applying Lemma \ref{SG}, we infer
\begin{align*}
	\frac{d}{ds}E_{2}^{\langle 1\rangle, (0)}(s)(s)
	&=
	-\int_{\R}f^2_{yy}\,\y
	+ \left( \theta+\frac{1}4 \right) \int_{\R}f_y^2 \, \y
	+ \left( 2\theta+\frac{1}2 \right) \frac{r^2e^{-s}}{a} \int_{\R}f_{y}g_y\y \\
	&\quad
	- \frac{e^{-s}}{a}\displaystyle\int_{\R}f^2_{yyy} \,\y
	+ \frac{r^2e^{-s}}{a} \int_{\R}g_y^2 \,\y 
	+ \left( \frac{r'}{a}-\frac{a'r}{a^2} \right) \int_{\R}f_{y}g_y \,\y \\
	&\quad
	+\int_{\R} h_yf_y \,\y
	+\int_{\R} k_yf_y \,\y.
\end{align*}
Integrating  by parts, 
we obtain the equality
\eqref{3.3b}, where
$\Lambda_{2}^{\langle 1\rangle, (0)}(s)$ is given by 
\begin{align*}
	\Lambda_{2}^{\langle 1\rangle, (0)}(s)
	&=
	-\frac12\int_{\R}f^2_{yy} \,\y
	+ \left( \theta+\frac{1}4 \right) \int_{\R}f_y^2 \,\y
	+ \left( 2\theta+\frac{1}2 \right) \frac{r^2e^{-s}}{a} \int_{\R}f_{y}g_y \,\y \\
	&\quad
	+ \frac{r^2e^{-s}}{a} \int_{\R}g_y^2 \,\y
	+ \left( \frac{r'}{a}-\frac{a'r}{a^2} \right) \int_{\R}f_{y}g_y \,\y
	-\int_{\R} hf_{yy} \,\y -\int_{\R} kf_{yy} \,\y.
\end{align*}
Applying the Schwarz inequality and
the estimates \eqref{v1}, \eqref{v2}, \eqref{h}, and  
\eqref{h02fk}, we get \eqref{3.300Ab}. 
 This concludes the proof of Lemma \ref{L6}.
\end{proof}

%
%
%
\begin{lem}\label{L7}
Let
$E_1^{\langle 1 \rangle, (0)}(s)$
be defined by \eqref{3.2bisb}.
Then
$E_1^{\langle 1 \rangle, (0)} \in{\mathcal C}^1([s_{0},s_{0}+S])$
and there exists
$C_{\flat}>0$
such that, for all
$s\in [s_{0},s_{0}+S]$,
\begin{equation}\label{A0}
	\frac{d}{ds}  E_1^{\langle 1 \rangle, (0)} (s) =
	- \frac{1}{2}\int_{\R}g_y^{2} \,{\mathrm{d}}y
	+ \Lambda_{1}^{\langle 1\rangle, (0)}(s),
\end{equation}
where,
\begin{equation}\label{01A0}
	\Lambda_{1}^{\langle 1\rangle, (0)}(s)
	\le
	C_{\flat} \int_{\R}(f^2+f_y^2+f^2_{yy}) \,\y
	+ C_{\flat} \frac{ e^{- s}}{a} \int_{\R}f^2_{yyy} \,\y
	+ C_{\flat} e^{-\mu s}.
\end{equation}
\end{lem}
\begin{proof}
The functional
$E_{1}^{\langle 1 \rangle,(0)}$
is of class
${\mathcal C}^1([s_{0},s_{0}+S])$.
Similarly to the proof of Lemma \ref{L6}, by applying Lemma \ref{LG}, we obtain
%
%
\begin{multline}
	\frac{d}{ds} E_{1}^{\langle 1\rangle, (0)}(s)
	= - \int_{\R}g_y^2 \,\y + \left( \theta+\frac34 \right) \int_{\R}f_{yy}^2 \,\y
	 + \left( \theta+\frac{3}4 \right) \frac{e^{-s}}{a} \int_{\R}f^2_{yyy} \,\y\nonumber\\
	+ \left( \theta+\frac{3}4 \right) \frac{r^2e^{-s}}{a} \int_{\R}g_y^2 \,\y 
	- \frac{a'r}{2a^2} \int_{\R}g_y^2 \,\y
	- \frac{a'}{2ra^2} \int_{\R}f_{yyy}^2 \,\y
	+ \int_{\R}h_yg_y \,\y
	+ \int_{\R}k_yg_y \,\y.
\end{multline}
Hence, we deduce the equality \eqref{A0}, where
$\Lambda_{1}^{\langle 1 \rangle, (0)}(s)$
is given by 
\begin{multline}
	\Lambda_{1}^{\langle 1\rangle, (0)} (s)
	=
	- \frac12 \int_{\R}g_y^2 \,\y
	+ \left( \theta+\frac34 \right) \int_{\R}f_{yy}^2 \,\y
	+ \left( \theta+\frac{3}4 \right) \frac{e^{-s}}{a} \int_{\R}f^2_{yyy} \,\y\nonumber\\
	+ \left( \theta+\frac{3}4 \right) \frac{r^2e^{-s}}{a} \int_{\R}g_y^2 \,\y 
	- \frac{a'r}{2a^2} \int_{\R}g_y^2 \,\y
	- \frac{a'}{2ra^2} \int_{\R}f_{yyy}^2 \,\y
	+ \int_{\R}h_yg_y \,\y
	+ \int_{\R}k_yg_y \,\y.
\end{multline}
We then infer from the Schwarz inequality,
the estimates \eqref{v1}, and \eqref{v2}, that
\begin{equation}\label{fg1}
	\Lambda_{1}^{\langle 1\rangle, (0)} (s)
	\le
	C \int_{\R}f_{yy}^2 \,\y
	+ C\frac{e^{-s}}{a} \int_{\R}f^2_{yyy} \,\y
	+ \int_{\R}(h_y^2+k_y^2) \,\y.
\end{equation}

Finally, using  \eqref{fg1}, \eqref{hbis}, \eqref{kprime}, \eqref{k1bb}  and \eqref{k1b},  we deduce 
\eqref{01A0}. 
 This concludes the proof of Lemma \ref{L7}.
 \end{proof}

\subsection{A priori estimate}
Putting all the above energy estimates together,
we reach the following a priori estimate for $(f,g)$.

\begin{prop}\label{Prop11}
We suppose that
$\kappa \delta_{0}$ is small enough (depending only
on $ p$),  $s_0>0$
large enough.
There exist a positive constant
$K_0>0$ and $\mu_0>0$ such that,
for any solution  $ (f,g)\in {\mathcal C}^{0}([s_{0},s_{0}+S],Z^1(\R))$  of (\ref{scaling})
satisfying (\ref{3.1}) with
$(f_{0},g_{0})=(f(s_{0}),g(s_{0})) \in Z^1(\R)$,
chosen
so that
$\Phi (e^{-s_{0}},f_{0},g_{0})<\delta_{0}^2$,
we have, for all
$s\in[s_0,s_{0}+S]$,
\begin{align}\label{P66}
	&\Phi (s,f(s),g(s))
	+ \int_{s_0}^s e^{-\mu_0(s-\tau)}
		\left(
		\|g(\tau)\|^2_{L^{2.1}(\R)} + \|f_{yy}(\tau) \|_{L^2(\R)}^2+\|g_{y}(\tau)\|_{L^2(\R)}^2
		\right) \,d\tau \\
	&\quad \le
	K_0 \left( \Phi (s_0,f(s_0),g(s_0)) +e^{-\mu s_0} \right) e^{-\mu_{0}(s-s_{0})}.
\end{align}
\end{prop}
\begin{proof}
Clearly 
$\E_{\vartheta, \zeta, \omega} \in{\mathcal C}^1([s_{0},s_{0}+S])$.
Moreover, by combining
Lemmas \ref{L03}, \ref{L5}, \ref{L6}, \ref{L7}, we obtain
\begin{align*}\label{EE0}
	\frac{d}{ds} \E_{\vartheta, \zeta, \omega  }(s)
	&\le
	- \frac{1}{2} \left( \frac{3}{4} - \theta \right) \int_{\R}(1+y^2)f^{2} \,{\mathrm{d}}y
	-\int_{\R}(1+y^2)f_{y}^{2} \,{\mathrm{d}}y
	-\frac{e^{-s}}{a}\int_{\R}(1+y^{2})f_{yy}^2 \,{\mathrm{d}}y \\
	&\quad
	-\frac{\vartheta}{2}\int_{\R} (1+y^{ 2})g^{2} \,{\mathrm{d}}y
	-\frac{\zeta e^{-s}}{a}\int_{\R}f_{yyy}^2 \,{\mathrm{d}}y
	-\frac{\zeta}2 \int_{\R} f_{yy}^2 \,{\mathrm{d}}y
	-\frac{\omega}{2}\int_{\R}g_y^{2} \,{\mathrm{d}}y\nonumber\\
	&\quad
	- \frac{4 \vartheta e^{-s}}{a} \int_{\R}yf_{yy} g_y \,\y
	+ \Lambda_{\rho_0}(s) + \Lambda_1^{(0,1)}
	+ \Lambda_1^{\langle 1\rangle, (0)}(s) + \Lambda_2^{\langle 1\rangle, (0)}(s).
\end{align*}
Here, the remainder terms satisfy
\begin{align*}
	&\Lambda_{\rho_0}(s) + \Lambda_1^{(0,1)}
	+ \Lambda_1^{\langle 1\rangle, (0)}(s) + \Lambda_2^{\langle 1\rangle, (0)}(s) 
	\\
	&\le
	C_{\ast\ast} \vartheta \frac{ e^{-s}}{a} \int_{\R}(1+y^2)f^2_{yy} \,\y
	\\
	&\quad
	+(C_{\ast} (\kappa\delta_{0})^{\overline {p}-1}
		+ C_{\ast\ast} \vartheta + C_{\sharp} \zeta +C_{\flat} \omega )
	\int_{\R} (1+y^2)f^2 \,\y
	\\
	&\quad 
	+ (C_{\ast} + C_{\sharp} \zeta )e^{-\mu s} \int_{\R}(1+y^2)g^2 \,{\mathrm{d}}y
	\\
	&\quad
	+ (C_{\ast\ast} \vartheta + C_{\sharp} \zeta + C_{\flat} \omega )
		\int_{\R}(1+y^{ 2})f^2_{y} \,{\mathrm{d}}y
	+ C_{\flat} \omega  \int_{\R}f^2_{yy} \,\y
	\\
	&\quad
	+ C_{\flat} \omega \frac{ e^{- s}}{a} \int_{\R}f^2_{yyy} \,\y
	+ (C_{\ast} + C_{\ast\ast} \vartheta + C_{\sharp} \zeta + C_{\flat} \omega)e^{-\mu s}.
\end{align*}
We apply the Schwarz inequality to obtain
\[
	-\frac{4 \vartheta e^{-s}}{a} \int_{\R}yf_{yy} g_y \,\y
	\le
	\frac{  e^{-s}}{4a}\displaystyle\int_{\R}y^2f^2_{yy} \,\y
	+ \frac{16 \vartheta^2 e^{-s}}{ a}\int_{\R} g_y^2 \,\y
\]
%
We now choose $\kappa \delta_0$, $\vartheta_0,$ $\omega_0$, and  $\zeta_0$ 
small enough,
and then choose $s_{0}$ large enough to deduce
\begin{align}\label{EE11v1}
	\frac{d}{ds} {\E}_{\vartheta_0,\zeta_0, \omega_0}(s) 
 	&\le
	-\frac{\lambda_0}2\int_{\R}(1+y^2)f^{2} \,{\mathrm{d}}y
	-\frac12\int_{\R}(1+y^2)f_{y}^{2}{\mathrm{d}}y
	\\
	&\quad
	-\frac{e^{-s}}{2a}\int_{\R}y^2f_{yy}^2 \,{\mathrm{d}}y
	-\frac{\vartheta_0}{4}\int_{\R} (1+y^{ 2})g^{2} \,{\mathrm{d}}y
	\\
	&\quad
	-\frac{\zeta_0}4\frac{ e^{-s}}{a}\int_{\R}f_{yyy}^2 \,{\mathrm{d}}y
	-\frac{\zeta_0}4\int_{\R} f_{yy}^2 \,{\mathrm{d}}y
	-\frac{\omega_0}{4}\int_{\R}g_y^{2} \,{\mathrm{d}}y
	+Ce^{-\mu s}
\end{align}
with some positive constants
$\lambda_0, \vartheta_0, \zeta_0, \omega_0$.
Additionally, we prove easily that there exists $C_{0}>1$ such that,
for all $s\in[s_{0},s_{0}+S]$, where $s_{0}$ is large enough,
\begin{equation}
\label{3.60}
\frac{1}{C_{0}} \Phi(s,f(s),g(s)) \le {\E}_{\vartheta_0,\zeta_0, \omega_0}(s) 
 \le
C_{0} \Phi(s,f(s),g(s)).
\end{equation}
We then conclude from  \eqref{EE11v1},
and (\ref{3.60}) that, there exist  $\mu_0<\mu$, small enough such that
\begin{equation}
\label{EE111}
\frac{d}{ds}  {\E}_{\vartheta_0,\zeta_0, \omega_0}(s)  +
\mu_0 {\E}_{\vartheta_0,\zeta_0, \omega_0}(s) + 
\mu_0\Big(\|g\|^2_{L^{2.1}(\R)}
+ \|f_{yy}\|_{L^2(\R)}^2+\|g_{y}\|_{L^2(\R)}^2\Big)
\le Ce^{-\mu s}.
\end{equation}
Integrating (\ref{EE111}) over $[s_0,s]$, we obtain,
for all $s\in [s_0,s_{0}+S]$,
\begin{eqnarray}
\label{EEE}
 {\E}_{\vartheta_0,\zeta_0, \omega_0}(s) +\mu_0
\int_{s_0}^s e^{-\mu_0(s-\tau)}\Big(\|g\|^2_{L^{2.1}(\R)}
+ \|f_{yy}\|_{L^2(\R)}^2+\|g_{y}\|_{L^2(\R)}^2\Big) d\tau
\\
\le
  \Big( {\E}_{\vartheta_0,\zeta_0, \omega_0}(s_0) +\frac{A_4}{\mu-\mu_0} e^{-\mu s_0}
\Big)\
e^{-\mu_{0}(s-s_{0})}.
\end{eqnarray}
Now (\ref{P66}) is a direct consequence of (\ref{3.60})
and (\ref{EEE}).  This concludes the proof of Proposition \ref{Prop11}.
\end{proof}
\subsection{Proof of Theorem \ref{T1} }
 We choose  $s_0>0$
large enough, and
$\delta_0$ small enough,
 if $(f_0,g_0)\in Z^{1}(\R)$ satisfies
$\|(f_0,g_0)\|_{Z^{1}(\R)} \le{\delta}_{0} $,
 then (\ref{scaling})
has a unique local solution
$(f,g)\in {  C}^0([s_0,S_{max}),Z^{1}(\R) )$
satisfying $ (f(s_{0}),g(s_{0}))=(f_{0},g_{0}) $.

To prove that this solution is global,
we argue by contradiction.
Assume that
there exists $\widetilde{S}>0$
such that
\begin{equation}
\label{A.mmm}
\|f(s)\|^2_{H^{1.1}(\R)}
+\|f_{yy}(s)\|^2_{L^2(\R)}
<\kappa^2 \delta^2_{0},\qquad \
\forall s \in [s_0,s_0+\widetilde{S}),
\end{equation}
and
\begin{equation}
\label{3.11m}
\|f(\widetilde{S})\|^2_{H^{1.1}(\R)}
+\|f_{yy}(\widetilde{S})\|^2_{L^2(\R)}
=\kappa^2 \delta_{0}^2.
\end{equation}
If   $s_0$ is large enough,  so that
$e^{-\mu s_{0}}\le
\delta_0^2$, we have by
(\ref{P66}) 
\begin{equation}
\label{P660}
 \Phi (s,f(s),g(s))  \le
K_0  \Big( \Phi (s_0,f(s_0),g(s_0)) + \delta_0^2
\Big), \qquad \forall s\in [s_0,s_0+\widetilde{S}].
\end{equation}

By using the definition of $\Phi (s,f(s),g(s)) $ given by \eqref{defphi} we infer 
$\Phi (s_0,f(s_0),g(s_0)) \le{\delta}_{0}^2 $. Therefore,   if $\kappa>\sqrt{8K_{0}}$, we get
\begin{equation}
\label{3.1s}
\|f(s)\|^2_{H^{1.1}(\R)}
+\|f_{yy}(s)\|^2_{L^2(\R)}
\le 2K_0 \delta_{0}^2\le \frac{{\kappa}^2
\delta_{0}^2}{4}, \qquad \forall s\in [s_0,s_0+\widetilde{S}],
\end{equation}
which  contradicts (\ref{3.11m}).
Thus, we have,
\begin{equation}
\|f(s)\|^2_{H^{1.1}(\R)}
+\|f_{yy}(s)\|^2_{L^2(\R)}
\le 
 {\kappa}^2 \delta^2_{0},\quad
\qquad \forall ~
t\in [s_0,s_{0}+S_{max}).
\end{equation}
By Proposition \ref{Prop11}, we conclude that
\begin{equation}\label{contra}
\Phi(s,f(s),g(s))
\le \frac{{\kappa}^2\delta^2_0}{4}
\le \frac{1}{4},\qquad \forall ~ s\in [s_0,s_{0}+S_{max}).
\end{equation}
Then, the solution can be continued
to $[s_0,\infty)$.  By Proposition \ref{Prop11},
the property (\ref{P123}) holds for any $s \in [s_0, \infty)$.  This concludes the proof of Theorem \ref{T1}.

\subsection{Proof of Theorem \ref{T2}}
In the  original variables $x$ and $t$, we first set
\begin{equation}\label{eq:utilde}
	\tilde u(t, x)=u(t,x)-\g(t, x),
\end{equation}
where $\g$ is defined by \eqref{def:gamma}.
The function $\tilde u$ is a solution of the following  equation:
\begin{equation}
\label{ori0}
	\partial_{tt}\tilde u +b(t) \partial_{t} \tilde u -a(t) \partial_{xx}\tilde u + \partial_{xxxx}\tilde u
		=-\left( |\tilde u+\g|^{p-1}(\tilde u+\g)-|\g|^{p-1}\g \right) + \mathcal{R} (t, x)
\end{equation}
for $t\in (t_0,\infty),  x\in \R$,
where
\begin{equation}\label{def:R}
	\mathcal{R}(t, x)
	= -\partial_{tt}\g -\big(b(t)-b_0(1+t )^{\beta} \big) \partial_{t} \g +\left( a(t)-(1+t )^{\alpha} \right)
		\partial_{xx}\g - \partial_{xxxx}\g.
\end{equation}
Moreover, the initial data at  $t=t_0$  is given by
\begin{equation}\label{ori0data}
	\tilde u_0=\tilde u(t_0, x)=u(t_0,x)-\g(t_0, x), \quad
	\tilde u_1=\tilde u_t(t_0, x)=u_t(t_0, x)-\g_t(t_0, x), \quad
	x\in\R.
\end{equation}
Assume that these initial data are sufficiently small in
$Z^{1}(\R)$,
that is,
$\| ( \tilde{u}_0, \tilde{u}_1 ) \|_{Z^{1}(\R)} \le \tilde{\delta}_{0}$
with $\delta_0$ determined in Theorem \ref{T1}.
Then, after the change of variables \eqref{scaling0},
the initial data $(f_0(y), g_0(y)) := (f(s_0,y), g(s_0,y))$
at the time
$s=s_{0} :=\log (1+R(t_0))$
are given by
\begin{align*}\label{1.16}
	f_0(y)
	&:=
	v(s_0,y) - \Omega(y) \\
	&=
	\frac{e^{\theta  s_{0}}}{A_0}u(R^{-1}(e^{s_0}-1), ye^{\frac{s_{0}}{2}})-\f(y)
\end{align*}
and
\begin{align*}
	g_0(y)
	&:=
	w(s_0,y) + \theta \Omega (y) + \frac{y}{2} \Omega'(y) \\ 
	&=\frac{e^{(\theta+1)s_0}}{A_0 r\left( R^{-1}(e^{s_0}-1) \right)}
		\partial_tu \left( R^{-1}(e^{s_0}-1), ye^{\frac{s_{0}}{2}} \right)
	+\theta  \f(y)+{\frac{y}{2}}\f'(y).
\end{align*}
From the assumption of Theorem \ref{T2},
we can apply Theorem \ref{T1} and obtain
\begin{align}
	\left\| u(t) - \Gamma(t) \right\|_{L^{\infty}}
	&\le
	C (R(t)+1)^{-\theta} 
	\left\| f ( \log (R(t)+1), y) \right\|_{L^{\infty}} \\
	&\le C (R(t)+1)^{-\theta} \left\| f ( \log (R(t)+1), y) \right\|_{H^{1}_y} \\
	&\le C (R(t)+1)^{-\theta-\mu_0/2}.
\end{align}
This completes the proof of Theorem \ref{T2}.

\appendix
\section{Auxiliary lemmas for energy estimates}

\subsection{Preliminary weighted energy identities}
Let  $l, m \in \R$, $n \in \mathbb{N}\cup \{0\}$,
and for functions $c_1, c_2, c_3$,
 we consider a system for two functions $f=f(s, y), g=g(s, y)$ given by
\begin{equation}\label{SG}
\left\{
	\begin{alignedat}{2}
	&g=f_{s}-\frac{y}{2}f_y-l  f,\\
	&c_1(s)\Big(g_s-\frac{y}2g_y-mg \Big)+c_2(s)g
	+g=-c_3(s)f_{yyyy}+f_{yy}+h
	\end{alignedat}
\right.
\end{equation}
in $(s, y) \in I \times \mathbb{R}$ with an interval $I \subset \mathbb{R}$,
where
$h$  is a given function belonging to ${\mathcal C}^0( I, H^{0,n}(\R))$.

We introduce the following
functionals:
\begin{equation}\label{energ0}
	\begin{alignedat}{2}
	E^{(n)}_{1}(s) &= \frac{1}{2} \int_{\R}y^{2n}
	\left( f_y^{2} +
	c_3(s)f_{yy}^{2} +c_1(s)g^{2} \right)  \y,\\
	E^{(n)}_{2}(s) &= \frac{1}{2} \int_{\R} y^{2n} \left( f^{2}
	+ c_1(s) fg \right) \y.
	\end{alignedat}
\end{equation}
\begin{lem}\label{LG}
Assume that
$c_1, c_3 \in \mathcal{C}^1(I)$, $c_2 \in \mathcal{C}(I)$, and 
$(f, g) \in \mathcal{C}(I, H^{2,n}(\R) \times H^{0,n}(\R))$
satisfy \eqref{SG} in the sense of distribution.
Then, we have
\begin{align*}
	\frac{d}{ds}E^{(n)}_{1}(s)
	&=
	- \int_{\R}y^{2n}g^2\y + \left( l-\frac{2n-1}4 \right) \int_{\R}y^{2n}f_y^2 \,\y
		+ \left( l-\frac{2n-3}4 \right) c_3(s) \int_{\R}y^{2n}f^2_{yy} \,\y \\
	&\quad
		+\left( m-\frac{2n+1}4 \right) c_1(s) \int_{\R}y^{2n}g^2 \,\y
		+\left( \frac{c'_1(s)}2-c_2(s) \right) \int_{\R}y^{2n}g^2 \,\y \\
	&\quad
		-2n\int_{\R} y^{2n-1} f_{y}g \,\y
		-4n c_3(s) \int_{\R}y^{2n-1}f_{yy}g_y \,\y \\
	&\quad
		-2n (2n-1)c_3(s) \int_{\R}y^{2n-2}f_{yy} g \,\y
		+\frac{c_3'(s)}2 \int_{\R}y^{2n}f_{yy}^2 \,\y
		+ \int_{\R}y^{2n} h g \,\y
\end{align*}
and
\begin{align*}
	\frac{d}{ds}E^{(n)}_{2}(s)
	&=
		- \int_{\R}y^{2n}f^2_{y} \,\y + \left( l-\frac{2n+1}4 \right) \int_{\R}y^{2n}f^2 \,\y
		+ \left( m+l-\frac{2n+1}2 \right) c_1(s) \int_{\R}y^{2n} fg \,\y\\
	&\quad - c_3(s) \int_{\R}y^{2n}f^2_{yy} \,\y + c_1(s) \int_{\R}y^{2n}g^2 \,\y
		-c_2(s) \int_{\R}y^{2n}fg \,\y \\
	&\quad -2n(2n-1)c_3(s)\int_{\R}y^{2n-2} f_{yy}f \,\y
		+2n (2n-1)c_3(s) \int_{\R}y^{2n-2}f^2_{y} \,\y\\
	&\quad +n(2n-1)\int_{\R} y^{2n-2}f^2 \,\y +c_1'(s) \int_{\R}y^{2n}fg \,\y  
		+\int_{\R} y^{2n} hf \, \y.
\end{align*}
\end{lem}
\begin{proof}
When $f, g$ are smooth, the proof is given by Lemma 3.1 in \cite{our1}.
Hence, we give a justification under the regularity assumption
$(f, g) \in \mathcal{C}(I, H^{2,n}(\R) \times H^{0,n}(\R))$.
We note that from the equation \eqref{SG}, we have
$f \in \mathcal{C}^1(I, L^2(\mathbb{R}))$
and $g \in \mathcal{C}^1(I, H^{-2}(\mathbb{R}))$.

Let $\rho_{\varepsilon}$ be a usual mollifier, that is,
$\rho_{\varepsilon}(y) := \varepsilon^{-1} \rho (y / \varepsilon )$
with a small parameter
$\varepsilon > 0$
and a function
$\rho \in \mathcal{C}_0^{\infty}(\mathbb{R})$
satisfying
$\rho \ge 0$, $\int_{\mathbb{R}} \rho(y) \,dy = 1$, and $\mathrm{supp\,} \rho \in [-1, 1]$.
Taking the convolution in the $y$-variable, we define
$f_{\varepsilon} := \rho_{\varepsilon} \ast f$,
$g_{\varepsilon} := \rho_{\varepsilon} \ast g$
and
$h_{\varepsilon} := \rho_{\varepsilon} \ast h$.
Then, we have
\begin{align*}
	\left\| (1+|y|^n) \int_{\mathbb{R}} \rho_{\varepsilon} (y-z) f (z) \,dz \right\|_{L^2}
	&\le
	C \left(
	\left\| \int_{\mathbb{R}} (1+|y-z|^n) \rho_{\varepsilon} (y-z) f (z) \,dz \right\|_{L^2} \right. \\
	&\quad
	\left.
	+ \left\| \int_{\mathbb{R}} \rho_{\varepsilon} (y-z) (1+|z|^n) f (z) \,dz \right\|_{L^2}
	\right) \\
	&\le
	C \| (1+|y|^n) \rho_{\varepsilon} \|_{L^1}
		\| (1+|y|^n) f \|_{L^2} \\
	&\le
	C \| (1+|\varepsilon y|^n ) \rho \|_{L^1} \| (1+|y|^n) f \|_{L^2} \\
	&\le 
	C \| f \|_{H^{0,n}},
\end{align*}
where we note that the constant $C$ can be taken uniformly in $\varepsilon > 0$.
Therefore, we see that $f_{\varepsilon}$ has the same weighted integrability as $f$, and hence,
$f_{\varepsilon} \in \mathcal{C}(I, H^{4,n}(\mathbb{R}))$
and
$g_{\varepsilon} \in \mathcal{C}(I, H^{2,n}(\mathbb{R}))$.
Moreover,
$(f_{\varepsilon}, g_{\varepsilon})$ satisfy the equations
\begin{equation*}
\left\{
	\begin{alignedat}{2}
	&g_{\varepsilon}=f_{\varepsilon,s}-\frac{y}{2}f_{\varepsilon,y} - l f_{\varepsilon}
	- \left[ \rho_{\varepsilon} \ast, \frac{y}{2}\partial_y \right] f, \\
	&c_1(s)\Big(g_{\varepsilon,s} - \frac{y}2 g_{\varepsilon,y} - mg_{\varepsilon} \Big)+c_2(s)g_{\varepsilon}
		+ g_{\varepsilon} =-c_3(s)f_{\varepsilon, yyyy}+f_{\varepsilon, yy}+ h_{\varepsilon}
		+ \left[ \rho_{\varepsilon} \ast, \frac{y}{2}\partial_y \right] g,
	\end{alignedat}
\right.
\end{equation*}
where
$[\rho_{\varepsilon} \ast, \frac{y}{2}\partial_y] f$
denotes the commutator
\[
	\left[ \rho_{\varepsilon} \ast, \frac{y}{2}\partial_y \right] f
	= \rho_{\varepsilon} \ast \left( \frac{y}{2} \partial_y f \right)
		- \frac{y}{2} \partial_y \left( \rho_{\varepsilon} \ast f \right).
\]
Here, we remark that
$\left[ \rho_{\varepsilon} \ast, \frac{y}{2}\partial_y \right] f \in H^{2,n}(\mathbb{R})$
and
$\left[ \rho_{\varepsilon} \ast, \frac{y}{2}\partial_y \right] g \in H^{0,n}(\mathbb{R})$
can be proved in the same way as in the unweighted case
(see, e.g., Mizohata \cite[Chapter 6]{Mi}).
Thus, from the equations, we also have
$f_{\varepsilon} \in \mathcal{C}^1(I, H^{2,n}(\mathbb{R}))$
and
$g_{\varepsilon} \in \mathcal{C}^1(I, H^{0,n}(\mathbb{R}))$.
Since $(f_{\varepsilon}, g_{\varepsilon})$ have enough regularity,
we can apply the conclusion of Lemma 3.1 of \cite{our1}.
For $\frac{d}{ds} E_1^{(n)}(s)$, we obtain the same identity as in the statement of Lemma \ref{LG}
with additional terms
\begin{align*}
	\int_{\mathbb{R}} y^{2n} \left( \left[ \rho_{\varepsilon}\ast, \frac{y}{2} \partial_y \right] f \right)_y f_y \,\mathrm{d}y
	+ c_3(s) \int_{\mathbb{R}} y^{2n} \left( \left[ \rho_{\varepsilon}\ast, \frac{y}{2} \partial_y \right] f \right)_{yy} f_{yy} \,\mathrm{d}y
	\\
	+ c_1(s) \int_{\mathbb{R}} y^{2n} \left( \left[ \rho_{\varepsilon}\ast, \frac{y}{2} \partial_y \right] g \right) g \,\mathrm{d}y
\end{align*}
in the right-hand side.
Now, we let $\varepsilon \to 0$ in the identity of $\frac{d}{ds} E_1^{(n)}(s)$.
We note that
$f_{\varepsilon} \to f$ in $H^{2,n}(\mathbb{R})$,
$g_{\varepsilon} \to g$ in $H^{0,n}(\mathbb{R})$,
$h_{\varepsilon} \to h$ in $H^{0,n}(\mathbb{R})$
(they can be proved in completely the same way as the unweighted case, see e.g., Racke \cite[Lemma 4.2]{Ra}),
and 
$\left[ \rho_{\varepsilon} \ast, \frac{y}{2}\partial_y \right] f \to 0$
in $H^{2,n}(\mathbb{R})$,
$\left[ \rho_{\varepsilon} \ast, \frac{y}{2}\partial_y \right] g \to 0$ in $H^{0,n}(\mathbb{R})$
(see e.g., Mizohata \cite[Chapter 6]{Mi})
hold.
Consequently, letting $\varepsilon \to 0$, we obtain the desired identity for $\frac{d}{ds} E_1^{(n)}(s)$.
The identity for $\frac{d}{ds} E_2^{(n)}(s)$ can be obtained in the same way.
\end{proof}

\subsection{Asymptotic behavior of the self-similar solution $\Omega$}
%
%
In this subsection, we recall the existence of self-similar solution $\Omega$ for the equation \eqref{1.4}:
\[
	\Omega''(z) + \frac{z}{2} \Omega'(z) + \theta \Omega(z) - |\Omega(z)|^{p-1} \Omega(z) = 0, \quad z \in \mathbb{R},
\]
and prove the asymptotic behavior of its higher order derivatives.
\begin{lem}\label{L3bis1}
For any $c_0 > 0$, there exists a smooth, even, positive solution $\Omega$ for \eqref{1.4}
satisfying the following.
\begin{align}\label{1.5}
	\f (z) &= c_0{|z|}^{-2\theta } + c_1 |z|^{-2\theta-2 \varsigma} + \smallO{|z|^{-2\theta-2 \varsigma} }, \\
\label{1.5000}
	\theta \f(z) + \frac{z}2\f'(z) &= c_2|z|^{-2\theta-2\varsigma}
	+\smallO{ |z|^{-2\theta-2\varsigma}},
\end{align}
with some constants $c_1, c_2$, where
\begin{equation}\label{1.5bis}
	\varsigma= \min (1,\theta (p-1)).
\end{equation}
Moreover, with some constants $c_6, c_7, c_8, c_9$, we have
\begin{align}\label{28oc2bis}
	\f^{''}(z) &= c_6 |z|^{-2\theta-2} +\smallO{|z|^{-2\theta-2}}, \\
\label{28oc2bisbis}
	\f^{'''}(z) &= c_8|z|^{-2\theta-3} +\smallO{|z|^{-2\theta-3}}, \\
\label{28oc2L}
	\left( \theta+\frac12 \right)  \f'(z)+\frac{z}{2}\f^{''}(z) &= c_7|z|^{-2\theta-2\varsigma-1}
	+\smallO{ |z|^{-2\theta-2\varsigma-1}}, \\
\label{28oc2kk}
	\f^{''''}(z) &= c_9|z|^{-2\theta-4}
	+\smallO{|z|^{-2\theta-4}}.
\end{align}

\end{lem}
\begin{proof}
The existence of a smooth, even, positive solution $\Omega$ to \eqref{1.4}
satisfying \eqref{1.5} and \eqref{1.5000}
was already proved by
Brezis, Peletier, and Terman \cite{BPT}
and Escobedo, Kavian, and Matano \cite{EKM}.
Therefore, it suffices to show \eqref{28oc2bis}--\eqref{28oc2kk}.
Consider now, for all $z\ge0,$
\begin{equation}\label{F0}
F_0(z)=\frac{\Omega'(z)}{\Omega(z)}.
\end{equation}
Note that $F_0$ is well-defined and differentiable. Moreover,  by 
using \eqref{1.4}, we get
\begin{equation}\label{28oc10}
F_0'(z)+\frac{z}2F_0(z)=-\theta+\Omega(z)^{p-1}-F_0(z)^2 :=g_0(z).
\end{equation}
Multiplying \eqref{28oc10} by $e^{\frac{z^2}4}$, and integrating over $(0,z),$  
we obtain
\begin{equation}\label{g0}
	F_0(z)=e^{-\frac{z^2}4}\int_0^zg_0(t)e^{\frac{t^2}4} \,\mathrm{d}t.
\end{equation}
Let us introduce the function $G_0$ defined by: 
\begin{equation}\label{G000}
G_0(z)= z^2\Big(\frac{z}2F_0(z)+\theta-\Omega^{p-1}(z)\Big).
\end{equation}
Thanks to  \eqref{g0}, we write
\begin{equation}G_0(z)
=\frac{\int_0^zg_0(t)e^{\frac{t^2}4}dt+\frac{2\theta}{z}{e^{\frac{z^2}4}}-\frac{2}{z}{e^{\frac{z^2}4}}\Omega^{p-1}(z)}{\frac2{z^3}e^{\frac{z^2}4}}
.\end{equation}
By L'Hopital's rule, we get 
$$\lim_{z\to \infty}G_0(z)=\lim_{z\to \infty}\frac{g_0(z)-\frac{2\theta}{z^2}+\theta+\frac{2}{z^2}\Omega^{p-1}(z)-\Omega^{p-1}(z)
-\frac{2p-2}{z}\Omega^{p-2}(z)\Omega'(z)}{\frac1{z^2}-\frac2{z^4}}.$$
Remember the definition of $g_0$ given by \eqref{28oc10}, we observe that
$$g_0(z)+\theta-\Omega^{p-1}(z)=-F_0^2(z).$$ Therefore,
\begin{equation}\label{G0}
\lim_{z\to \infty}G_0(z)=-2\theta+\lim_{z\to \infty}(-z^2F^2_0(z)+2\Omega^{p-1}(z)
-(2p-2)z\Omega^{p-2}(z)\Omega'(z)).
\end{equation}
Employing   \eqref{1.5}, and
\eqref{1.5000},  we deduce 
  $$ \lim_{z\to \infty}zF_0(z)=-2\theta,\ \   \textrm{and} \ 
 \lim_{z\to \infty}\Omega^{p-1}(z)=\lim_{z\to \infty}
z\Omega^{p-2}(z)\Omega'(z)=0.$$
  Consequently, we obtain that
\begin{equation}\label{G0b}
\lim_{z\to \infty}G_0(z)=-2\theta-4\theta^2.
\end{equation}
By using the expression $G_0$ given  by 
\eqref{G000} and the fact that $\f$ satisfy \eqref{1.4}, 
we  infer 
\begin{equation}\label{omega2}
z^2\f^{''}(z)
=-G_0(z)\Omega(z).
\end{equation}
 Thus, by  \eqref{1.5},  \eqref{G0b},  and \eqref{omega2}, we have
for $|z|$ large enough,
\begin{equation}\label{3.29tt}
 \f^{''}(z)=(2\theta+4\theta^2)|z|^{-2\theta-2}
+\smallO{
|z|^{-2\theta-2}}.\quad\qquad
\end{equation}
This proves \eqref{28oc2bis}.

Let $F_1(z)=\frac{\Omega''(z)}{\Omega'(z)}$, for $z>0$.
  In view of \eqref{1.4}, it is easy to see that
\begin{equation}\label{28oc1}
	F_1'(z)+\frac{z}2F_1(z)
	=
	- \left(\theta+\frac12 \right) + p\Omega^{p-1}(z)-F_1^2(z) := g_1(z).
\end{equation}
Multiplying \eqref{28oc1} by $e^{\frac{z^2}4}$, and integrating over $(1,z),$  
we obtain
\begin{equation}\label{e1}
	F_1(z) =
	F_1(1)e^{-\frac{z^2}4}+e^{-\frac{z^2}4}\int_1^zg_1(t)e^{\frac{t^2}4} \,\mathrm{d}t.
\end{equation}
We now introduce the following function:
\begin{equation}\label{Gf1}
	G_1(z)
	= \left( \frac{z}2F_1(z)+\theta+\frac12-p\f^{p-1}(z) \right) z^{2}.
\end{equation}
Thanks to  L'Hopital's rule, and the identity \eqref{e1}, we conclude 
\begin{multline}
	\lim_{z\to \infty}G_1(z)
	=
	\lim_{z\to \infty} \frac12 z^{3} e^{-\frac{z^2}4}
	\left( \int_{z_0}^zg_1(t)e^{\frac{t^2}4} \,\mathrm{d}t
		+(2\theta +1)z^{-1}e^{\frac{z^2}4} - 2p\f^{p-1}(z) z^{-1}e^{\frac{z^2}4} \right)
	\\
	=\lim_{z\to \infty}
	\frac{g_1(z)e^{\frac{z^2}4}
	+(2\theta+1)(z^{-1}e^{\frac{z^2}4})'
	-2p(\f^{p-1}(z) z^{-1}e^{\frac{z^2}4})'}{(\frac2{z^{3}}e^{\frac{z^2}4})'}.
\end{multline}
Hence, by
recalling the definition of $g_1$ given by \eqref{28oc1}, we have
\begin{multline}
\lim_{z\to \infty}G_1(z)
=\lim_{z\to \infty}
\frac{-F^2_1(z)-\frac{2\theta+1}{z^2}
-2p(p-1)\f^{p-2}(z)\f'(z)
z^{-1}
+2p\f^{p-1}(z)
z^{-2}}{
-\frac6{z^4}+\frac1{z^2}}\\
=\lim_{z\to \infty}
\big(-z^2F^2_1(z)-(2\theta+1)-
2p(p-1)z\f^{p-2}(z)\f'(z)
+2p\f^{p-1}(z)\big).
\end{multline}
Employing   \eqref{1.5}, 
\eqref{1.5000},  and \eqref{28oc2bis} we deduce 
  $$ \lim_{z\to \infty}zF_1(z)=c_{10},\ \   \textrm{and} \ 
 \lim_{z\to \infty}\Omega^{p-1}(z)=\lim_{z\to \infty}
z\Omega^{p-2}(z)\Omega'(z)=0,$$
for some $c_{10}\in \R$.  Consequently, we obtain that
\begin{equation}\label{l1}
\lim_{z\to \infty}G_1(z)
=c_{11},
\end{equation}
for some $c_{11}\in \R$.
Also,  by differentiating  the identity   \eqref{1.4},  we infer that
\begin{equation}\label{defo}
	\f^{'''}(z) =
	-\left( \left( \theta+\frac12 \right) \f'(z)+\frac{z}{2}\f^{''}(z) \right)
	+ p\f^{p-1}(z)\f'(z).
\end{equation}
By using the expression $G_1$ given  by 
\eqref{Gf1} and the identity \eqref{defo}, 
we  infer 
\begin{equation}\label{omega2}
z^2\f^{'''}(z)
=-G_1(z)\Omega'(z).
\end{equation}
Thus, by \eqref{1.5},  \eqref{1.5000}, \eqref{l1}, and \eqref{omega2}, we have
for $|z|$ large enough,
\begin{equation}\label{3.29ttr}
 \f^{'''}(z)=c_{12}|z|^{-2\theta-3}
+\smallO{
|z|^{-2\theta-3}}.\quad\qquad
\end{equation}
Thus \eqref{28oc2bisbis} holds.

We are now ready to prove the estimate \eqref{28oc2L}. In fact, by \eqref{1.4} we obtain
\begin{equation}\label{n1}
(\theta+\frac12)  \f'(z)+\frac{z}{2}\f^{''}(z)=-
\f^{'''}(z)
+p\f^{p-1}(z)\f'(z).
\end{equation}

The result \eqref{28oc2L}  follows immediately from 
 \eqref{1.5},
 \eqref{1.5000},  and  \eqref{3.29ttr}, and the definition of $\upsilon$ given by \eqref{1.5bis}.

Now, we focus on the remaining estimate.
Let us first mention that  by  differentiating  the identity   \eqref{defo},  we infer that
\begin{equation}\label{dds}
\f^{''''}(z)=-\big((\theta+1)  \f''(z)+\frac{z}{2}\f^{'''}(z)\big)
+p|\f(z)|^{p-1}\f''(z) +p(p-1)\f(z)^{p-2}(\f'(z))^2.
\end{equation}
Since  $\f''$ satisfies  \eqref{28oc2bis}, it follows that  there  exists $z_0>0$, such that $ \Omega''(z)>0$, for all $z\ge z_0$.
Let 
\begin{equation}
F_2(z)=\frac{\Omega'''(z)}{\Omega''(z)},\quad  \forall z\ge z_0.
\end{equation}
We clearly have   
\begin{equation}\label{28oc1b}
F_2'(z)+\frac{z}2F_2(z)=-(\theta+1)+p\Omega^{p-1}(z) +p(p-1)\f(z)^{p-2}\frac{(\f'(z))^2}{\f''(z)}
-F_2^2(z):=g_2(z).
\end{equation}
Furthermore, by multiplying \eqref{28oc1b} by $e^{\frac{z^2}4}$, and integrating over $(z_0,z),$  
we obtain
\begin{equation}\label{reF2}
F_2(z)=F_2(z_0)e^{-\frac{z^2}4}+e^{-\frac{z^2}4}\int_{z_0}^zg_2(t)e^{\frac{t^2}4}dt.
\end{equation}
Finally, we introduce 
\begin{equation}\label{28oc1bg2}
G_2(z)= z^2\Big(\frac{z}2F_2(z)+\theta+1-p\Omega^{p-1}(z)
-p(p-1)\f^{p-2}(z)\frac{(\f'(z))^2}{\f''(z)}\Big).
\end{equation}
By exploiting \eqref{reF2},  we trivially write the following:
\begin{equation}\label{refF22}
	\lim_{z\to \infty}G_2(z)
	=
	\frac12\lim_{z\to \infty}
	\frac{\Big(\int_{z_0}^zg_2(t)e^{\frac{t^2}4} \,\mathrm{d}t+\frac2{z}e^{\frac{z^2}4}\big(\theta +1
	-p\Omega^{p-1}(z)
	-p(p-1)\f^{p-2}(z)\frac{(\f'(z))^2}{\f''(z)}\big) \Big)}{z^{-3}e^{\frac{z^2}4}}.
\end{equation}
By L'Hopital's rule and     
  the definition of $g_2$ given by \eqref{28oc1b}, we have
\begin{multline}
\lim_{z\to \infty}G_2(z)
=\lim_{z\to \infty}
\frac{-F_2^2(z)-2(\theta+1)\frac1{z^2}
+2p\Omega^{p-1}(z)\frac1{z^2}
-2p(p-1)\Omega^{p-2}(z)\f'(z)\frac1{z}
}{
-\frac{6}{z^4}+\frac1{
z^2}}\\
-p(p-1)\frac{(\frac1{z}\f^{p-2}(z)(\f'(z))^2\frac1{\f''(z)})'}{
-\frac{3}{z^4}+\frac1{2
z^2}}.
\end{multline}
Clearly, by using the above, we obtain
\begin{equation}\label{l1bis}
\lim_{z\to \infty}G_2(z)=-2\theta-2-\lim_{z\to \infty}
z^2F^2_2(z)=c_{13}.
\end{equation}
  By using the expression $G_2$ given  by 
\eqref{28oc1bg2} and the identity \eqref{dds}, 
we  infer 
\begin{equation}\label{omega2bis}
z^2\f^{''''}(z)
=-G_2(z)\Omega''(z).
\end{equation}
 Therefore, by  \eqref{28oc2bis}, \eqref{l1bis},  and \eqref{omega2bis}, we have
for $|z|$ large enough,
\begin{equation}\label{3.29ttrbb}
 \f^{''''}(z)=c_{12}|z|^{-2\theta-4}
+\smallO{
|z|^{-2\theta-4}}.\quad\qquad
\end{equation}
Thus  \eqref{28oc2kk} holds.
This concludes the proof of lemma \ref{L3bis1}.
\end{proof}

\subsection{Weight function}
As mentioned before, an appropriate choice for the weight function $q(y)$ can be, for example,
as follows
\begin{equation}\label{3.28}
	q(y)=
	\begin{dcases}
		e^{\frac{y^2}{4}} \int_{y}^{+\infty} e^{-\frac{r^2}{4}}
		\left( \f ^{-2}(r) - \f^{-2}(0) \right) \,{\mathrm{d}}r
		&(y \ge 0), \\
		e^{\frac{y^2}{4}} \int^{y}_{-\infty} e^{-\frac{r^2}{4}}
		\left( \f ^{-2}(r) - \f^{-2}(0) \right) \,{\mathrm{d}}r
		&( y \le 0 ).
	\end{dcases}
\end{equation}
%
%

We are now going to prove  the  above
listed properties of $q$.
\begin{lem}\label{L3}
The function $q$ defined by (\ref{3.28})
is  a positive even function of class
${\mathcal C}^2(\R)$ satisfying the following differential inequality:
\begin{equation}\label{3.25}
	q''(y) + \left( \frac{2 \f' }{\ \f }-\frac{y}{2} \right) q'(y)
	- \left( \frac{y \f' }{\ \f }+\frac{1}{2} \right) q(y)
	\le 0 \qquad  \text{in} \quad \R.
\end{equation}
Moreover, the function
$q$ satisfies (\ref{3.34}) and (\ref{3.35})
as well as the property
\begin{equation}
\label{3.26}
\qquad\qquad\qquad q(y)=C
|y|^{4\theta -1}+{\mathcal   O}
(|y|^{4\theta -3}),\qquad\qquad |y|\to +\infty.
\end{equation}
Furthermore,  $q$ 
 satisfying the bounds
\begin{equation}\label{3.34}
C^{-1}(1+|y|^{4\theta -1})
\le q(y) \le
C (1+|y|^{4\theta -1}), \quad
y\in \R,
\end{equation}
\begin{equation}
\label{3.35}
|yq'(y)| \le C (1+|y|^{4\theta -1}),
\quad y\in \R,
\end{equation}
\begin{equation}
\label{3.35bis}
|q''(y)| \le C (1+|y|^{4\theta -1}),
\quad y\in \R,
\end{equation}
 where $C>1$.
\end{lem}
\begin{proof}
Using \eqref{1.5},
we can write, for $|z|$ large enough,
\begin{equation}\label{3.29}
 \f^{-2}(z)=\frac{1}{c_0^2}
{|z|}^{4\theta }
+{\mathcal O}(|z|^{4\theta -2\upsilon}),
\end{equation}
 which implies that the function $q$,
given by   (\ref{3.28}), is
well defined.
Moreover, by using the fact that, we have
$y \f' (y)<0$, for all $y\in \R$,
it is easy to see that $q(y)$ is a solution
of the differential
 inequality (\ref{3.25}). Since
$\ \f $ is an even function,  $q$
 is an even function
of class ${\mathcal C}^2(\R)$.  Therefore, to  prove that  the estimates 
 \eqref{3.34}, \eqref{3.35},  
 \eqref{3.35bis} hold, 
we just need to obtain   similar  estimates  for $x$ large enough.

Now, let us show the proof in the case $x \geq 0$, because the proof in $x \leq 0$ follows in a similar manner.
Multiplying (\ref{3.29})
by $e^{-\frac{z^2}{4}}$, integrating over
$[x,\infty)$ and finally multiplying
by $e^{\frac{x^2}{4}}$ we have,
 for $x$ large enough,
\begin{align}\label{3.30}
	q(x)
	&=
	e^{\frac{x^2}{4}}\int_{x}^{+\infty} e^{\frac{-z^2}{4}}
		\left( \f ^{-2}(z)-\f^{-2}(0) \right) \,{\mathrm{d}} z
\nonumber \\
	&=
	\frac1{c_0^2}e^{\frac{x^2}{4}}\int_{x}^{+\infty} e^{\frac{-z^2}{4}}
	z^{4\theta } \,{\mathrm{d}}z
	+{\mathcal O} \left( e^{\frac{x^2}{4}} \int_{x}^{+\infty}
		e^{\frac{-z^2}{4}} z^{4\theta -2\upsilon}\,{\mathrm{d}}z \right).
\end{align}
Integrating by parts, we obtain, for $x$ large enough,
\begin{align}\label{3.32}
	e^{\frac{x^2}{4}} \int_{x}^{+\infty} ze^{\frac{-z^2}{4}} z^{4\theta -1} \,{\mathrm{d}}z
	&=
	2x^{4\theta -1}+{\mathcal O}
	\left( e^{\frac{x^2}{4}}
		\int_{x}^{+\infty} e^{\frac{-z^2}{4}} z^{4\theta -2\upsilon} \,{\mathrm{d}}z \right).
\end{align}
Hence, for $x$ large enough, we have
\begin{equation}\label{3.33}
	q(x) = \frac2{c_0^2}x^{4\theta -1}+{\mathcal O} \left( \frac{q(x)}{x^{2\upsilon}} \right)
	= \frac2{c_0^2}x^{4\theta -1}
	+{\mathcal O} \left( x^{4\theta -1-2\upsilon} \right).
\end{equation}
By exploiting the fact $q$ is even and the fact $q$ is positive over $\R$, we deduce that the estimate \eqref{3.34} holds. \\
Next,  by differentiating  the function $q$ given by \eqref{3.28}, we infer that
\begin{equation}\label{qdiff}
	q'(x)=\frac{x}2q(x)- (\f ^{-2}(x)-\f^{-2}(0)),
	\qquad\forall x \ge 0. 
\end{equation}
Furthermore, by writing 
$q(x)=-2e^{\frac{x^2}{4}}\int_{x}^{\infty}
-\frac{z}2e^{\frac{-z^2}{4}}
\frac{\ \f ^{-2}(z)-\f^{-2}(0)}{z}{\mathrm{d}}z$,
and integrating by parts, we obtain 
\begin{equation}\label{3.32bb}
	q(x) = \frac{2( \f ^{-2}(x)-\f^{-2}(0))}{x}-2e^{\frac{x^2}{4}}\int_{x}^{\infty} e^{\frac{-z^2}{4}}
	\frac{2z \f ^{-3}(z)\f'(z)+\f^{-2}(z)-\f^{-2}(0)}{z^2}{\mathrm{d}}z.
\end{equation}
Inserting \eqref{3.32bb} into
\eqref{qdiff}, yields
\begin{equation}\label{3.32bbb}
q'(x) = - x e^{\frac{x^2}{4}}\int_{x}^{\infty}
e^{\frac{-z^2}{4}}
\frac{2z \f ^{-3}(z)\f'(z)+\f^{-2}(z)-\f^{-2}(0)}{z^2}{\mathrm{d}}z.
\end{equation}
Hence, by \eqref{1.5}, and  \eqref{1.5000}, we write
for $x$ large enough
\begin{equation}\label{3.33bv}
q'(x)= {\mathcal O}(x^{4\theta -2}).
\end{equation}
Clearly \eqref{3.35} directly follows from \eqref{3.33bv}.
Also,  by differentiating  the identity   \eqref{qdiff},  we infer that
\begin{equation}
\label{qdiff1}
q''(x)=\frac{x}2q'(x)+\frac12 q(x)+2\f ^{-3}(x)\f'(x),
\qquad\forall x \ge 0. 
\end{equation}
Gathering all the above results, namely  \eqref{qdiff1}, \eqref{3.34}, \eqref{3.35},  \eqref{1.5}, and  \eqref{1.5000},  we end up with the
estimate \eqref{3.35bis}.
This concludes the proof of lemma \ref{L3}.
\end{proof}


\subsection{Estimates of remainder terms}
In this subsection, we collect the estimates of remainder terms appearing in the energy estimates
discussed in Section 2.

First, we prepare the following estimates for the functions involving
$a$ and $b$ that are frequently used in the energy estimates.
Noting that
\begin{align}
    \frac{d}{ds} r(t(s))
    = \frac{d}{ds} r(R^{-1}(e^s-1))
    = \frac{r'(t(s))}{r(t(s))} e^s,\label{rp}\\
    \frac{d}{ds} a(t(s))
    = \frac{d}{ds} a(R^{-1}(e^s-1))
    = \frac{a'(t(s))}{r(t(s))} e^s, \label{ap}
\end{align}
and using $a=br$, we have
\begin{align}\label{eq:c3prime}
	\frac{d}{ds} \left( \frac{e^{-s}}{a} \right)
	&= -\frac{a'}{ra^2}-\frac{e^{-s}}{a}
\end{align}
and
\begin{align}
\label{eq:c1prime}
    \frac{d}{ds} \left( \frac{r^2e^{-s}}{a} \right)
    &=
    \frac{1}{a^2}
    \left( 2 r \frac{dr}{ds} a e^{-s}
    -r^2 a e^{-s}
    -r^2 \frac{da}{ds} e^{-s}
    \right) \\
    &= \frac{2r'}{a} - \frac{r^2 e^{-s}}{a}- \frac{ra'}{a^2}
     = \frac{r'}{a} - \frac{r^2 e^{-s}}{a}- \frac{b'}{b^2}.
\end{align}
The following lemma gives the estimate for the function \eqref{eq:c1prime}.
\begin{lem}\label{lem:remainder}
Set
\begin{align}
    \mu &:=
    \min \left\{ \frac{\beta + 1}{\alpha-\beta + 1}, \frac{1}{\alpha-\beta + 1},
    \frac{2\alpha-\beta + 1}{\alpha-\beta+1} \right\},
\end{align}
which is positive if
$(\alpha,\beta) \in \mathcal{A}$.
Then, we have
\begin{align}
    \frac{r^2 e^{-s}}{a}
    \sim 
        e^{-\frac{\beta+1}{\alpha-\beta+1}s}
    \le
    C e^{-\mu s},
    \quad
    \frac{e^{-s}}{a}
    \sim
        e^{-\frac{2\alpha-\beta+1}{\alpha-\beta+1}s}
    \le
    C e^{-\mu s},\label{v1}
\end{align}
and
\begin{align} 
    \left| \frac{ra'}{a^2} \right|+
  \left| \frac{b'}{b^2} \right|+    \left| \frac{r'}{a} \right|
    \le C \frac{r^2e^{-s}}{a}.\label{v2}
\end{align}
\end{lem}

\begin{proof} See \cite{YW1}.
\end{proof}

Next, we give the estimates for the remainder term $h(s,y)$ defined in \eqref{1.17}.
\begin{lem}\label{lem:est:h}
Let $h(s,y)$ be defined in \eqref{1.17}.
Under \textbf{Assumptions (I), (II), (III)}, and \eqref{3.1}, we have
\[
	\| h \|_{H^{0,1}} \le C e^{-\mu s}.
\]
\end{lem}
\begin{proof}
From \eqref{1.17}, we have $h = h_1 + h_2$.
Exploiting \eqref{1.5}, \eqref{1.5bis}, \eqref{1.5000}, \eqref{v1}, and \eqref{v2}, we infer
\begin{equation}\label{L200}
	\|(1+y^2)^{\theta +\upsilon}h_1\|_{L^{\infty}(\R)} \le C e^{- \mu s}.
\end{equation}
Now, we  should prove that
\begin{equation}\label{h01}
\| h_1 \|_{L^{2,1}(\R)} \le C e^{- \mu s}.
\end{equation}
We distinguish three cases:

\noindent
\textbf{(Case 1):}
First case ($\theta (p-1)\le \frac14$):
 Remembering the fact that, by \textbf{Assumption (III)}, we have
 $p <\frac1{1-\frac{4(1-\beta )}{3(\alpha-\beta+1)}},$   namely   
$\theta p>\frac3{4}$. By
using \eqref{L200}, we conclude estimate 
 \eqref{h01} holds in this case.

\noindent
\textbf{(Case 2):}
Second case ($\frac14<\theta (p-1)\le 1$):
Observing $\theta>\frac12$ from \textbf{Assumption (II)}, we have
\[
	\theta p=\theta (p-1)+\theta>\frac14+\frac12=\frac34.
\]
Then, the estimate  \eqref{h01} holds in this case.

\noindent
\textbf{(Case 3):}
Third case ($\theta (p-1)> 1$):
Observe that,  $\theta>\frac12$. Therefore, by the use of \eqref{L200}, the estimate 
\eqref{h01} holds in this case.

Next,   by using the expression of $h_2$ given by \eqref{1.17}, we conclude easily that
\begin{align}\label{h002}
	\|h_2\|_{L^{2,1}(\R)}
	&\le
		C |\varepsilon (t)| \big(\||f(s)|^p\|_{{L^{2,1}(\R)}}+\|\f^p\|_{ L^{2,1}(\R)}\big) \\
\nonumber
	&\le
		C |\varepsilon (t)| \big(  \|f(s)\|_{ L^{\infty}}^{p-1}\|f(s)\|_{{L^{2,1}(\R)}}+\|\f^p\|_{ L^{2,1}(\R)}\big).
\end{align}
By the Sobolev inequalities and the assumption (\ref{3.1}), we have
\begin{equation}\label{kjo1}
	\|f(s)\|_{ L^{\infty}(\R)}\le C \|f(s)\|_{H^{1,0}(\R)}\le C.
\end{equation}
Proceeding similarly as for the term $h_2$,  and using
\textbf{Assumption (III)}, we get 
\begin{equation}\label{h0021}
	\|\f^p\|_{ L^{2,1}(\R)} \le C.
\end{equation}
In view of \eqref{h002}, (\ref{3.1}), \eqref{kjo1}, \eqref{h0021}, we obtain
\begin{equation}\label{h02f}
	\|h_2\|_{L^{2,1}(\R)} \le C |\varepsilon (t)|\le C e^{- \mu s}
\end{equation}
Therefore, by using the fact that  $h=h_1+h_2$,  the estimates  \eqref{h01} and \eqref{h02f}, we infer
\begin{equation}\label{h}
	\|h\|_{L^{2,1}} \le C e^{- \mu s}.
\end{equation}
This completes the proof.
\end{proof}

In the next two lemmas, we estimate the nonlinear terms of (\ref{scaling}).
\begin{lem}\label{lem:est:k}
Let $k(s,y)$ be defined in \eqref{defk}.
Under \textbf{Assumptions (I), (II), (III)} and \eqref{3.1}, $k$ satisfies
\begin{equation}\label{h02fk}
	\|k\|_{L^{2,1}(\R)} \le C \| f\|_{L^{2,1}(\R)}.
\end{equation}
\end{lem}
\begin{proof}
Using the basic inequality 
\[
	\big||a+b|^{p-1}(a+b)-a^p\big| \le C \big( |a|^{p-1}+|b|^{p-1}\big) |b|,\quad
	\forall  a\ge 0, \quad b\in \R,
\]
we conclude easily that
\begin{equation*}\label{h002k}
	\|k\|_{L^{2,1}(\R)} \le C
	\big(  \|f(s)\|_{ L^{\infty}}^{p-1}+\|\f\|^{p-1}_{ L^{\infty}(\R)}\big)\|f(s)\|_{{L^{2,1}(\R)}}.
\end{equation*}
Gathering the above two estimates with \eqref{kjo1}, and (\ref{3.1}),
and the fact that $\f\in L^{\infty}(\R)$, we arrive at the conclusion.
\end{proof}


\begin{lem}\label{LN}
 Assume that $(f,g)\in {\mathcal  C}^0([s_{0},s_{0}+S],Z^{1}(\R))$
is a solution of (\ref{scaling}) satisfying
(\ref{3.1}).
Then, there exist
$\delta_{1}=(\kappa\delta_{0})^{\overline {p}-1}$
and $C_0>0$ such that 
\begin{equation}\label{3.37}
	\int_{\R}(1+y^{ 2})|f{\mathcal N}(f)| \,{\mathrm{d}}y
	\le C_0 \delta_{1}
	\int_{\R}(1+y^{ 2})f^2\,{\mathrm{d}}y,
\end{equation}
where $\overline{p}=\min (2,p)$.
\end{lem}
\begin{proof}
By using (\ref{1.17}), we obtain
\begin{equation}\label{3.38}
	\int_{\R}(1+y^{2})|f{\mathcal N}(f)|\,{\mathrm{d}}y\le
	C
 	\int_{\R}(1+y^{2})|f|^{\overline{p}+1}\,{\mathrm{d}}y.
\end{equation}
 On the other hand, by the Sobolev inequalities and (\ref{3.1})
we have, for all $s\in [s_{0},s_{0}+S]$
\begin{equation}\label{kjo}
	\|f(s)\|_{ L^{\infty}(\R)} \le C \|f(s)\|_{H^{1,0}(\R)} \le C \kappa\delta_{0}.
\end{equation}
The estimate (\ref{kjo}) together with (\ref{3.38}) implies that
\begin{equation}
\label{ee}
\int_{\R}(1+y^{ 2})|f{\mathcal N}(f)| \,{\mathrm{d}}y\le C
 \delta_{1}\int_{\R}(1+y^{ 2})f^2 \,{\mathrm{d}}y.
\end{equation}
This concludes the proof of  Lemma \ref{LN}.
\end{proof}

Finally, we prove the estimates for the derivatives $h_y$ and $k_y$.
\begin{lem}
Let $h$ and $k$ be defined by \eqref{1.17} and \eqref{defk}, respectively.
Then, we have
\begin{align}
\label{hbis}
	\left\| \frac{\partial h}{\partial y} \right\|_{L^2(\R)}
	&\le
	C e^{- \mu s},
	\qquad
	\left\| \frac{\partial k}{\partial y} \right\|_{L^2(\R)}
	\le C \| f \|_{H^{1,0}}.
\end{align}
\end{lem}
\begin{proof}
Now, by
 recalling the definition  of
$h_1$ given by  \eqref{1.17}, and
 exploiting \eqref{1.5},
\eqref{1.5bis}, \eqref{1.5000},
\eqref{v1}, and \eqref{v2},
we infer
\begin{equation*}\label{L200prime}
	\left\| (1+y^2)^{\theta +\upsilon}\frac{\partial {h_1}}{\partial y} \right\|_{L^{\infty}(\R)}
	\le
	C e^{- \mu s}.
\end{equation*}
Therefore
\begin{equation}\label{h01primek}
	\left\| \frac{\partial h_1}{\partial y} \right\|_{L^2(\R)} \le C e^{- \mu s}.
\end{equation}
Next, by using the expression of $h_2$ given by \eqref{1.17}, we conclude easily that
\begin{align*}
	\left\| \frac{\partial h_2}{\partial y} \right\|_{L^2(\R)} 
	&\le
	C |\varepsilon (t)|
	\left( \|f(s)\|_{ L^{\infty}}^{p-1} +\|\f\|_{ L^{\infty}}^{p-1} \right)
	\left( \|f_y(s)\|_{{L^2(\R)}}+\|\f'\|_{ L^2(\R)} \right).
\end{align*}
By  \eqref{kjo1} and (\ref{3.1}) and the fact that
$\|\f\|_{ L^{\infty}(\R)}+ \|\f'\|_{ L^2(\R)}\le C$,
we have
\begin{equation}\label{h02fbis}
	\left\| \frac{\partial h_2}{\partial y} \right\|_{L^2(\R)} 
	\le C |\varepsilon (t)|\le C e^{- \mu s}.
\end{equation}
Therefore, by using the fact that  $h=h_1+h_2$,  the estimates  \eqref{h01primek} and \eqref{h02fbis},
we infer \eqref{hbis}.

Next, differentiating \eqref{defk} yields
\begin{equation}\label{kprime}
k_y=\underbrace{-p|f+\ \f |^{p-1}f_y}_{k_1}
\underbrace{-p\f'\big(|f+\ \f |^{p-1}
-\f^{p-1}\big)}_{k_2}.
\end{equation}
Adding the fact that
$\|\f\|_{ L^{\infty}(\R)}\le C$,
and \eqref{kjo1}, we write
\begin{equation}\label{k1}
\|k_1 \|_{L^2(\R)} 
 \le C 
 \big(  \|f\|_{ L^{\infty}}^{p-1}
 +\|\f\|_{ L^{\infty}}^{p-1}\big)
 \|f_y\|_{{L^2(\R)}}\le C \|f_y\|_{{L^2(\R)}}.
\end{equation}
To control the term $\|k_2\|_{L^2}$, namely to obtain
\begin{equation}\label{k1bb}
\|k_2 \|_{L^2(\R)} 
 \le C  \|f\|_{{L^2(\R)}},
\end{equation}
we distinguish two cases:

\noindent
\textbf{(Case 1):} $p\ge 2$

By using the fundamental theorem of calculus  we  obtain
\begin{equation}\label{k1b}
\|k_2 \|_{L^2(\R)} 
 \le C 
 \big(  \|f\|_{ L^{\infty}}^{p-2}
 +\|\f\|_{ L^{\infty}}^{p-2}\big)
 \|f\|_{{L^2(\R)}}.
\end{equation}
Proceeding similarly as for the term $k_1$, namely, 
by combining    the fact that  $\|\f\|_{ L^{\infty}(\R)}\le C$,  and  \eqref{kjo1},  we deduce  \eqref{k1bb} holds.

\noindent
\textbf{(Case 2)}: $p< 2$

First,  Since $ \f (y) >0$, for all $y\in \R$,  we  write
\begin{equation}\label{k2}
	k_2
	=
	-p \f' \f^{p-1} \left( \left| 1+\frac{f}{\f} \right|^{p-1}-1\right)
	=
	-p\f' \f^{p-1} N_0 \left( \frac{f}{\f} \right),
\end{equation}
where $N_0:\R\to \R$ is 
 given by $N_0(X)=|1+X|^{p-1}-1$. Clearly  $N_0(X)= {\mathcal O}(|X|)$ as $X\to 0, $
 and  $N_0(X)= {\mathcal O}(|X|^{p-1})$ as $|X|\to \infty.$ Therefore,
\begin{equation}\label{N0}
	|N_0(X)|
	\le 
	\begin{dcases}
	C|X| 
	&(|X| \le 1),\\
	C|X|^{p-1}
	&(|X| \ge 1).
	\end{dcases}
\end{equation}
Estimate \eqref{N0}, and the fact $p<2,$   yield
\begin{equation}
\label{N00}
|N_0(X)|\le C|X|, 
\qquad\forall X\in \R.
\end{equation}
Hence, combining \eqref{N00} and \eqref{k2}, we infer
\begin{equation}\label{k2b}
\|k_2 \|_{L^2(\R)} 
 \le C   \|\f' \f^{p-2}f\|_{{L^2(\R)}}\le  C
  \|\f' \f^{-1}\|_{{L^{\infty}(\R)}}
   \| \f\|^{p-1}_{{L^{\infty}(\R)}}
    \|f\|_{{L^2(\R)}}.
\end{equation}
The result  \eqref{k1bb} follows immediately from  \eqref{1.5} and \eqref{1.5000},  and \eqref{k2b}.
Hence, the estimate \eqref{k1bb}  holds true, for the two cases $p\ge2$ and $p< 2.$
\end{proof}

\section{Local Existence of Solutions}
In this section, we prove the local existence of the solution to
the equation \eqref{ori0} for the perturbation $\tilde{u}$
with the initial condition \eqref{ori0data}
in the weighted space $Z^1(\mathbb{R})$.
Let $t_0 > 0$.
We define $U={}^t( \tilde{u}, \partial_t \tilde{u})$ and we reformulate  \eqref{ori0} as 
an abstract Cauchy problem in the Hilbert space
$\dot{\mathcal{H}}^2(\mathbb{R})\times L^2(\mathbb{R})$,
where $\dot{\mathcal{H}}^2(\mathbb{R})$ denotes the complesion of $C_0^{\infty}(\mathbb{R})$
with respect to the norm $\| u_{xx} \|_{L^2}$:
\begin{equation}\label{syst}
\frac{d}{dt}U(t)=AU(t)+ B(t,U(t)),
\quad
U(t_0) = U_0 = {}^t (\tilde{u}_0, \tilde{u}_1),
\end{equation}
where
\[
	A=\begin{pmatrix} 0 & I \\ -\partial_{xxxx} &  0 \end{pmatrix}, \quad
	B(t, U(t))= \begin{pmatrix} 0  \\ a(t)\partial_{xx} \tilde{u}-b(t) \partial_{t} \tilde{u}
			+ \mathcal{G}(\tilde{u}) + \mathcal{R} \end{pmatrix}
\]
with the domain
$D(A) = \{ (u, v) \in \dot{\mathcal{H}}^2(\mathbb{R}) \times L^2(\mathbb{R}) ;\,
v \in \dot{\mathcal{H}}^2(\mathbb{R}), u_{xxxx} \in L^2(\mathbb{R}) \}$,
where
\[
	\mathcal{G}(\tilde{u}) := - \left( |\tilde{u}+\Gamma|^{p-1} (\tilde{u}+\Gamma) - |\Gamma|^{p-1}\Gamma \right)
\]
and $\mathcal{R}$ is defined by \eqref{def:R}.
By a standard argument, we can see that
$A$ is skew-adjoint and generates a unitary group
$\mathcal{T}(t)=\exp ((t-t_0)A)$
on $\dot{\mathcal{H}}^2(\mathbb{R})\times L^2(\mathbb{R})$
(see e.g. \cite{CaHa} and \cite{Pa}).
This and a simple $L^2$-estimate for $t > t_0$
\[
	\| u(t) \|_{L^2} \le
	\| \tilde{u}_0 \|_{L^2} + \int_{t_0}^t \| v(\tau) \|_{L^2} \, \mathrm{d}\tau
	 \le C(1+(t-t_0)) \left( \| \tilde{u}_0 \|_{H^2} + \| \tilde{u}_1 \|_{L^2} \right)
\]
imply that
$\mathcal{T}(t)$ restricted on $C_0^{\infty}(\mathbb{R}) \times C_0^{\infty}(\mathbb{R})$
is extended to a continuous semigroup on $H^2(\mathbb{R}) \times L^2(\mathbb{R})$
with the estimate
$\| \mathcal{T}(t) U_0 \|_{H^2 \times L^2} \le C(1+(t-t_0)) \| U_0 \|_{L^2}$
for $t > t_0$.
Now, we introduce the definition of the mild solution.
\begin{Def}\label{def:sol}
Let $I = [t_0,T]$ with some $T>t_0$
or $I = [t_0,\infty)$.
We say that
a function $u$ 
is a mild solution to \eqref{ori0} on $I$,
if  $U= {}^t (\tilde{u}, \partial_t \tilde{u})$ satisfies
$U \in {\mathcal C}(I, H^2(\mathbb{R})\times L^2(\mathbb{R}))$ and
\begin{equation}\label{eq:sol}
    U(t) = \mathcal{T}(t) U_0 + \int_{t_0}^t \mathcal{T}(t-s) B(s; U(s)) \,\mathrm{d}s
    \quad \text{in} \ {\mathcal {C}}(I, H^2(\mathbb{R})\times L^2(\mathbb{R})).
\end{equation}
\end{Def}
 We start from the following local existence result in the original variables
\begin{lem}\label{local1}
Suppose \textbf{Assumptions (I), (II), (III)}.
\begin{itemize}
\item[(i)] \textup{(Existence, uniqueness and continuous dependence)}
For any  $U_0 = {}^t(\tilde{u}_{0}, \tilde{u}_{1}) \in  H^3(\R) \times H^1(\R) $,
there exist $T > 0$ and
a unique mild solution $\tilde{u}$ to the equation \eqref{ori0}
such that
$U= {}^t (\tilde{u}, \partial_t \tilde{u}) \in {\mathcal C}(I, H^3(\R) \times H^1(\R))$,
where $I = [t_0, T]$.
The solution $U$
depends continuously on the initial data in 
$\mathcal{C}( I, H^2(\R)\times L^2(\R))$.
\item[(ii)] \textup{(Blow-up alternative)}
Let $T_{\max}$ be the lifespan of the local solution defined by
\[
	T_{\max} := \sup \{ T \in (t_0, \infty) ;\,
	\exists !  U \in \mathcal{C}([t_0, T], H^3(\R)\times H^1(\R)) \text{ \textup{:} mild solution to \eqref{ori0}} \}.
\]
If $T_{\max} < + \infty$, then we have
\[
	\lim_{t \to T_{\max}-0} \| U (t) \|_{H^3 \times H^1} = + \infty. 
\]
\item[(iii)] \textup{(Boundedness of weighted norm)}
If $U_0=(\tilde{u}_0, \tilde{u}_1) \in Z^1(\R)$
and $U \in \mathcal{C}(I, H^3(\R) \times H^1(\R))$
is the corresponding mild solution to \eqref{ori0},
then, $U \in \mathcal{C} (I, Z^1(\R))$.
\end{itemize}
\end{lem}
\begin{proof}{}
\noindent
\textbf{(Step 1)}
Let $T_0 > t_0$ be arbitrary fixed.
Let $T \in (t_0, T_0)$ and $I = [t_0, T]$. 
Let us first show the existence of the solution in the space
$\mathcal{C}(I ; H^3(\R) \times H^1(\R))$.
Let $\delta$ be determined later and we define
\[
	\mathcal{M}_{\delta} := \{ U \in \mathcal{C}(I, H^3(\R) \times H^1(\R)) \,;\,
						\sup_{t\in I} \| U(t) \|_{H^3 \times H^1} \le 2 \delta \}.
\]
Then, $\mathcal{M}_{\delta}$ is a complete metric space with the metric
$d (U, V) = \sup_{t\in I} \| U(t) - V(t) \|_{H^2 \times L^2}$.
\par
For $U \in \mathcal{M}_{\delta}$, we define the mapping
\[
	\Psi[U] (t) := \mathcal{T}(t) U_0 + \int_{t_0}^t \mathcal{T}(t-s) B(s; U(s)) \,\mathrm{d}s,
	\quad t \in [t_0,T].
\]
Then, we will see that $\Psi$ is a contraction mapping on $\mathcal{M}_{\delta}$,
provided that $T$ is sufficiently close to $t_0$.
Indeed, since
$\mathcal{T}(t)$
is continuous on $H^2(\R) \times L^2(\R)$ and commutes with the derivative $\partial_x$,
it is continuous on $H^3(\R) \times H^1(\R)$
with the estimate
$\| \mathcal{T}(t) U_0 \|_{H^3 \times H^1} \le C_{T_0} \| U_0 \|_{H^3 \times H^1}$.
Thus, we have
\begin{align}\label{eq:Phi}
	\| \Psi[U] (t) \|_{H^3 \times H^1}
	&\le
	C_{T_0} \| U_0 \|_{H^3 \times H^1} \\
	&\quad
	+ C_{T_0} \int_{t_0}^t \| a(\tau) \partial_{xx} \tilde{u}(\tau) - b(\tau) \partial_t \tilde{u}(\tau)
				+\mathcal{G}(\tilde{u}) + \mathcal{R} \|_{H^1} \,\mathrm{d}\tau.
\end{align}
Now, we define $\delta := C_{T_0} \| U_0 \|_{H^3\times H^1}$.
In order to estimate the right-hand side, we first apply the Taylor expansion to obtain
\begin{align}\label{eq:Gtildeu}
	\mathcal{G}(\tilde{u})
	&=
	- \left( | \tilde{u} + \Gamma |^{p-1} (\tilde{u} + \Gamma) - |\Gamma|^{p-1} \Gamma \right)
	=
	- \tilde{u} \int_0^1 p | \eta \tilde{u} + \Gamma|^{p-1} \,\mathrm{d} \eta.
\end{align}
Then, noting
$| \eta \tilde{u} + \Gamma|^{p-1} 
\le (\| \tilde{u} \|_{L^{\infty}} + \| \Gamma \|_{L^{\infty}})^{p-1}
\le C_T ( \| \tilde{u} \|_{H^1} + 1)^{p-1}$
for $t \in I$,
we have
$\| \mathcal{G}(\tilde{u}) \|_{L^2} \le C_T \| \tilde{u} \|_{L^2} ( \| \tilde{u} \|_{H^1} + 1)^{p-1}
\le C_T \delta (1+\delta)^{p-1}$,
where, the constant
$C_T$
depends on
$\sup_{t \in I} \| \Gamma(t) \|_{L^{\infty}}$.
Here, we also remark that $\Gamma(t)$ does not belong to $L^2(\R)$ in general
(see \eqref{1.5}),
and this is the reason to consider the difference in the definition of $\mathcal{G}(\tilde{u})$.
Next, we estimate the derivative
$\partial_x \mathcal{G}(\tilde{u})$.
Noting
\begin{align*}
	\partial_x \mathcal{G}(\tilde{u})
	&= - p | \tilde{u} + \Gamma|^{p-1} (\partial_x \tilde{u} + \partial_x \Gamma)
	+ p |\Gamma|^{p-1} \partial_x \Gamma.
\end{align*}
and
$ \| \partial_x \Gamma (t) \|_{L^2} < \infty$ for $t \in [t_0, T]$,
which follows from Lemma \ref{L3bis1},
one can see that 
\begin{align*}
	\| \partial_x \mathcal{G}(\tilde{u}) \|_{L^2}
	&\le
	C \| \tilde{u} + \Gamma \|_{L^{\infty}}^{p-1}
		\left( \| \partial_x \tilde{u} \|_{L^2} + \| \partial_x \Gamma \|_{L^2} \right)
	+ C \| \Gamma \|_{L^{\infty}}^{p-1} \| \partial_x \Gamma \|_{L^2} \\ 
	&\le
	C_T \left( (1+\delta)^{p} + 1 \right).
\end{align*}
Moreover, by Lemma \ref{L3bis1},
$\| \mathcal{R} \|_{H^1}$
is easily estimated by a constant $C_T$.
Applying the above inequalities to \eqref{eq:Phi}, we have
\begin{align*}
	&\int_{t_0}^t \| a(\tau) \partial_{xx} \tilde{u}(\tau) - b(\tau) \partial_t \tilde{u}(\tau)
				+\mathcal{G}(\tilde{u}) + \mathcal{R} \|_{H^1} \,\mathrm{d}\tau \\
	&\le
	C_T \cdot (T-t_0) \left( \delta + \delta(1+\delta)^{p-1} + (1+\delta)^{p} + 1 \right),
\end{align*}
where $C_T$ depends on
$\sup_{t \in I} (\| \Gamma(t) \|_{L^{\infty}} + \| \partial_x \Gamma(t) \|_{L^{2}})$
and
$\sup_{t\in I} ( a(t) + b(t) )$.
Thus, the mapping $\Psi$ maps $\mathcal{M}_{\delta}$ into itself
if $T$ is sufficiently close to $t_0$ so that
$C_T \cdot (T-t_0) \left( \delta + \delta(1+\delta)^{p-1} + (1+\delta)^{p} + 1 \right) < \delta$
holds.

Next, we estimate
$d ( \Psi[U], \Psi[V] )$
for 
$U = {}^t (\tilde{u}, \partial_t \tilde{u}), V={}^t (\tilde{v}, \partial_t \tilde{v}) \in \mathcal{M}_{\delta}$.
In the same way as \eqref{eq:Phi}, we first have
\[
	\| \Psi[U](t) - \Psi[V](t) \|_{H^2 \times L^2}
	\le
	C_0 \int_{t_0}^t
		\| a(\tau) (\partial_{xx} \tilde{u} - \partial_{xx} \tilde{v})
			+ b(\tau) (\partial_{t} \tilde{u} - \partial_{t} \tilde{v})
			+ \mathcal{G}(\tilde{u}) - \mathcal{G}(\tilde{v}) \|_{L^2} \,\mathrm{d}\tau.
\]
The definition of $\mathcal{G}(\tilde{u})$ leads to
\[
	\mathcal{G}(\tilde{u}) - \mathcal{G}(\tilde{v})
	= - | \tilde{u} + \Gamma |^{p-1} (\tilde{u} + \Gamma)
	+  | \tilde{v} + \Gamma |^{p-1} (\tilde{v} + \Gamma)
	= p (\tilde{v}-\tilde{u}) \int_0^1
	\left| (1-\eta) \tilde{u} + \eta \tilde{v} + \Gamma \right|^{p-1} \,\mathrm{d}\eta.
\]
From this, we can show that for $U, V \in \mathcal{M}_{\delta}$,
\begin{equation}\label{eq:lwp:diff}
	\| \Psi[U] (t) - \Psi[V] (t) \|_{H^2 \times L^2}
	\le C_T T (1+\delta)^{p-1} \sup_{\tau\in I} \| U(\tau) - V(\tau) \|_{H^2 \times L^2},
\end{equation}
where $C_T$ depends on
$\sup_{t \in I} \| \Gamma(t) \|_{L^{\infty}}$
and
$\sup_{t\in I} ( a(t) + b(t) )$.
Therefore, if $T$ is sufficiently close to $t_0$ depending on $\delta$ and
$\sup_{t \in I} \| \Gamma(t) \|_{L^{\infty}}$, $\sup_{I} ( a(t) + b(t) )$,
then the mapping $\Psi$ is contractive on $\mathcal{M}_{\delta}$.
Here, we remark that if $p$ is close to $1$, then
the nonlinear function $u \mapsto |u|^{p-1}u$ is not of $\mathcal{C}^2$ and 
we cannot expect a good estimate for
$\| |u|^{p-1}u - |v|^{p-1}v \|_{H^1}$.
This is the reason why we used the $H^2 \times L^2$ norm for the metric of $\mathcal{M}_{\delta}$.
Consequently, the Banach fixed point theorem implies that there exists a mild solution
$U$ in
$\mathcal{C}(I ; H^3(\R) \times H^1(\R))$.

The uniqueness in
$\mathcal{C}(I ; H^3(\R) \times H^1(\R))$
immediately follows from
the estimate \eqref{eq:lwp:diff} above.
The continuous dependence on the initial data in
$\mathcal{C}(I ; H^2(\R) \times L^2(\R))$
is obtained in the same way as the uniqueness, and we omit the detail.
Moreover, the blow-up alternative follows from a standard contradiction argument.
\par
\noindent
\textbf{(Step 2)}
We prove the boundedness of weighted norm.
Let  $(\tilde{u}_0, \tilde{u}_1) \in Z^1(\R)$,
and
let $U$ be the corresponding mild solution on $[t_0,T]$
to the initial data $U_0$.
Let
$M := \sup_{t\in I} \| U(t) \|_{H^{3}(\R)\times H^{1}(\R)}$.
Note that, differentiating \eqref{eq:sol}, we see that
$U(t)$ satisfies the differential equation \eqref{syst} in
$\mathcal{C}(I, H^1(\R) \times H^{-1}(\R))$.

Now, we set 
\begin{align*}
    &\chi \in \mathcal{C}_0^{\infty} (\mathbb{R}),\quad
    0\le \chi \le 1, \quad
    \chi(x) = \begin{cases} 1 & \textrm {if} \ |x|\le 1,\\ 0 & \textrm {if}\ |x| \ge 2, \end{cases}\\
    &\chi_n(x) := \chi \left( \frac{x}{n} \right) \quad
    (n \in \mathbb{N}).
\end{align*}
By
$\mathrm{supp}\, \chi_n \subset [-2n, 2n]$,
we easily see that
\begin{align}
\label{eq:lwp:dxchi1}
    &| \partial_x (x^2 \chi_n(x)^2 ) |
    =
    \left| 2 x \chi_n(x)^2 + 2 \frac{x^2}{n} \chi' \left( \frac{x}{n} \right) \chi_n(x) \right|
    \le C |x| \chi_n(x),\\
\label{eq:lwp:dxchi2}
    &| \partial_{xx} (x^2 \chi_n(x)^2 ) | \\
\nonumber
    &\quad =
    \left|
    2 \chi_n(x)^2 + 4 \frac{x}{n} \chi' \left( \frac{x}{n} \right) \chi_n(x)
    + 2 \frac{x^2}{n^2} \left( \left( \chi'\left(\frac{x}{n}\right)\right)^2 + \chi''\left( \frac{x}{n} \right) \chi_n(x) \right)
    \right| \\
\nonumber
    &\quad \le C, 
\end{align}
with some constant $C>0$.
Consider
\begin{align*}
	E_n (t; \tilde{u})
	&:= \int_{\mathbb{R}} x^2 \chi_n(x)^2
	\left( |\partial_t \tilde{u}(t,x)|^2
		+ a(t) |\partial_x \tilde{u}(t,x)|^2
		+ |\partial_{xx} \tilde{u}(t,x)|^2
		+ |\tilde{u}(t,x)|^2 \right) \,\mathrm{d}x,\\
	E (t; \tilde{u})
	&:= \int_{\mathbb{R}} x^2
		\left( |\partial_t \tilde{u}(t,x)|^2
			+ a(t) |\partial_x \tilde{u}(t,x)|^2
			+ |\partial_{xx} \tilde{u}(t,x)|^2
			+ | \tilde{u}(t,x)|^2 \right) \,\mathrm{d}x.
\end{align*}
Note that $E_n(t; \tilde{u})$ is finite thanks to
$\chi_n$.
Differentiating it, we have
\begin{align*}
    \frac{d}{dt} E_n (t; \tilde{u})
    &= 2 \int_{\mathbb{R}} x^2 \chi_n(x)^2
    \left(  \partial_t \tilde{u} \partial_{tt} \tilde{u}
    + a(t) \partial_x \tilde{u} \partial_t \partial_x \tilde{u}
    + \partial_{xx} \tilde{u} \partial_t \partial_{xx} \tilde{u}
    \right) \,\mathrm{d}x \\
    &    + 2 \int_{\mathbb{R}} x^2 \chi_n(x)^2 \tilde{u} \partial_t \tilde{u} \,\mathrm{d}x
    	+ \int_{\mathbb{R}} x^2 \chi_n(x)^2 a'(t) |\partial_x \tilde{u} |^2 \,\mathrm{d}x.
\end{align*}
By the integration by parts and  using the equation \eqref{syst},
the right-hand side can be written as
\begin{align*}
    &2 \int_{\mathbb{R}} x^2 \chi_n(x)^2 \partial_t \tilde{u}
    \left( -b(t) \partial_t \tilde{u} + \mathcal{G}(\tilde{u}) + \mathcal{R} \right) \,\mathrm{d}x \\
    &\quad
    	-2 \int_{\mathbb{R}} \partial_x (x^2 \chi_n(x)^2 )
		a(t) \partial_x \tilde{u} \partial_t \tilde{u} \,\mathrm{d}x \\
    &\quad
    	+ 4 \int_{\mathbb{R}}  \partial_x (x^2 \chi_n(x)^2 )
		\partial_x^3 \tilde{u} \partial_t \tilde{u} \,\mathrm{d}x
    +2 \int_{\mathbb{R}}  \partial_{xx} (x^2 \chi_n(x)^2 )
    	\partial_x^2 \tilde{u} \partial_t \tilde{u} \,\mathrm{d}x \\
    &\quad    + 2 \int_{\mathbb{R}} x^2 \chi_n(x)^2 \tilde{u} \partial_t \tilde{u} \,\mathrm{d}x
    	+ \int_{\mathbb{R}} x^2 \chi_n(x)^2 a'(t) |\partial_x \tilde{u} |^2 \,\mathrm{d}x.
\end{align*}
Here, we remark that the integral including $\partial_x^4 u$ makes sense as
$\langle x^2 \chi_n(x)^2 \partial_t u, \partial_x^4 u \rangle_{H^1, H^{-1}}$
and the above computation is justified.
By using \eqref{eq:lwp:dxchi1} and \eqref{eq:lwp:dxchi2}
and applying the Schwarz inequality,
the last five terms and the term including $-b(t)\partial_t \tilde{u}$
can be further estimated by
\[
	C (2M)^2 + C_{T,a,b} E_n (t; \tilde{u})
\]
with some constant $C, C_{T,a,b} > 0$.
Moreover, thanks to the expression \eqref{eq:Gtildeu} and the Schwarz inequality,
the term including
$\mathcal{G}(\tilde{u})$
can be estimated by
\begin{align*}
	&C \left( \int_{\mathbb{R}} x^2 \chi_n(x)^2 |\partial_t \tilde{u}|^2 \,\mathrm{d}x \right)^{1/2}
	\left( \int_{\mathbb{R}} x^2 \chi_n(x)^2 \left( \| \tilde{u} \|_{L^{\infty}} + \| \Gamma \|_{L^{\infty}} \right)^{p-1}
		| \tilde{u} |^2 \,\mathrm{d}x \right)^{1/2} \\
	&\le C_{T, M, \Gamma} E_n(t; \tilde{u})
\end{align*}
with some
$C_{T, M, \Gamma} >0$.
Finally, in the same way as the derivation of the estimate \eqref{h},
which will be proved later,
it can be seen that the term including $\mathcal{R}$ is estimated by
\begin{align*}
	C \int_{\mathbb{R}} x^2 \chi_n(x)^2 |\partial_t \tilde{u}|^2 \,\mathrm{d}x
	+ C \int_{\mathbb{R}} x^2 \chi_n(x)^2 |\mathcal{R}|^2 \,\mathrm{d}x
	\le C_{T, \Gamma} (1+ E_n(t; \tilde{u}))
\end{align*}
with some $C_{T, \Gamma}>0$ independent of $n$.

Consequently, we obtain the estimate for $t \in I$,
\[
	\frac{d}{dt} E_n (t; \tilde{u})
	\le C_{T, a, b, M, \Gamma} (1+ E_n(t; \tilde{u}))
\]
and hence, the Gronwall inequality implies
\[
    E_n (t; \tilde{u} ) \le \tilde{C}_{T,a,b,M, \Gamma} E_n(t_0; \tilde{u})
\]
for $t \in I$,
where the constant $\tilde{C}_{T,a,b,M,\Gamma}$
is independent of $n$.
Then, letting $n \to \infty$, we conclude
\[
    E (t; \tilde{u} ) \le \tilde{C}_{T,a,b,M,\Gamma} E (t_0; \tilde{u}),
\]
which shows the boundedness of the norm $\| \tilde{u}(t), \partial_t \tilde{u}(t)) \|_{Z^1}$ 
for any $t \in I$.
The continuity of $( \tilde{u}(t), \partial_t \tilde{u}(t))$
with respect to the variable $t$ in the topology of $Z^1(\R)$
follows from the estimate
\begin{align}
    | E_n(t; \tilde{u} ) - E_n (s; \tilde{u}) |
    &\le
    \int_s^t \left| \frac{d}{d\sigma} E_n(\sigma ; \tilde{u}) \right| \,d\sigma 
    \le
    C_{T,a,b,M,\Gamma} (t-s)
\end{align}
for $s < t$
and taking the limit $n\to \infty$.
This proves
$( \tilde{u}, \partial_t \tilde{u}) \in \mathcal{C}(I, Z^1(\R))$
and concludes the proof of Lemma \ref{local1}.
\end{proof}

In the new variables $(s, y)$, Lemma \ref{local1} becomes:
\begin{lem}\label{local2}
Suppose \textbf{Assumptions (I), (II), (III)}.
\begin{itemize}
\item[(i)] \textup{(Existence, uniqueness and continuous dependence)}
 Let 
$\delta>0 $ be given.
There exists $ s_{\max}>0 $ such that, for all  $
(f_0,g_0) \in  H^{3}(\R)\times H^{1}(\R) $, with $ \|(f_{0}, g_{0})\|_{H^{3}\times H^{1}}<\delta $,
the equation \eqref{scaling} has a unique (mild) solution
 $ (f(s),g(s))
\in {\mathcal C}([s_0,s_{\max}),H^{3}(\R) \times H^{1}(\R)) $ with $
(f(s_0),g(s_0))=(f_0,g_0) $. The solution $
(f(s), g(s)) $ depends continuously on
the initial data in 
$\mathcal{C}([s_0, s_{\max}) , H^{2}(\R) \times L^2(\R))$.

\item[(ii)] \textup{(Blow-up alternative)}
Let $s_{\max}$ be the lifespan of the local solution defined by
\[
	s_{\max} := \sup \{ s \in [s_0, \infty) ;\,
	\exists ! (f, g) \in \mathcal{C}([s_0, s], H^3(\R)\times H^1(\R)) \ \text{: mild solution to \eqref{scaling}} \}.
\]
If $s_{\max} < + \infty$, then we have
\[
	\lim_{s \to s_{\max}-0} \| (f, g)(s) \|_{H^3 \times H^1} = + \infty. 
\]

\item[(iii)] \textup{(Boundedness of the weighted norm)}
If $(f_0, g_1) \in Z^1(\R)$ and
$(f,g) \in \mathcal{C}([s_0, s_{\max}), H^3(\R) \times H^1(\R))$
is the corresponding mild solution to \eqref{scaling}.
Then,
$(f,g) \in \mathcal{C} ([s_0, s_{\max}), Z^1(\R))$.

\end{itemize}
\end{lem}

\bibliographystyle{plain}

\end{document}